\documentclass[11pt]{amsart}
\usepackage{amsmath}
\usepackage{amssymb}
\usepackage{amsthm}
\usepackage{latexsym}
\usepackage{hyperref}
\usepackage{enumerate}
\usepackage{graphicx}
\usepackage[all]{xy}

\newtheorem{thm}{Theorem}[section]

\newtheorem{lem}[thm]{Lemma}
\newtheorem{prop}[thm]{Proposition}
\newtheorem{defn}[thm]{Definition}
\newtheorem{conj}{Conjecture}

\newcommand{\enuma}[1]{\begin{enumerate}[\textup{(}a\textup{)}] {#1} \end{enumerate}}

\newcommand{\mb}{\mathbf}

\newcommand{\mc}{\mathcal}
\newcommand{\mf}{\mathfrak}

\newcommand{\isom}{\xrightarrow{\;\sim\;}}
\newcommand{\mosi}{\xleftarrow{\;\sim\;}}

\newcommand{\N}{\mathbb N}
\newcommand{\Z}{\mathbb Z}
\newcommand{\Q}{\mathbb Q}
\newcommand{\R}{\mathbb R}
\newcommand{\C}{\mathbb C}

\newcommand{\matje}[4]{\left(\begin{smallmatrix} #1 & #2 \\ 
#3 & #4 \end{smallmatrix}\right)}

\newcommand{\q}{/\!/}

\def\Hom{{\rm Hom}}
\def\End{{\rm End}}

\def\Irr{{\rm Irr}}
\def\Gal{{\rm Gal}}
\def\GL{{\rm GL}}

\def\PGL{{\rm PGL}}
\def\SL{{\rm SL}}

\def\cG{{\mathcal G}}
\def\cS{{\mathcal S}}
\def\cT{{\mathcal T}}

\def\cH{{\mathcal H}}
\def\cB{{\mathcal B}}

\def\cR{{\mathfrak R}}
\def\cZ{{\mathcal Z}}
\def\Fr{{\rm Frob}}

\def\ab{{\rm ab}}

\def\Ind{{\rm Ind}}
\def\ind{{\rm ind}}

\def\nr{{\rm nr}}
\def\unr{{\rm unr}}
\def\unr{{\rm unr}}
\def\rU{{\rm U}}

\def\fs{{\mathfrak s}}

\def\Rep{{\rm Rep}}

\def\Res{{\rm Res}}

\def\der{{\rm der}}
\def\ad{{\rm ad}}
\def\sc{{\rm sc}}
\def\Aut{{\rm Aut}}
\def\temp{{\rm temp}}
\def\cpt{{\rm cpt}}

\def\cusp{{\rm cusp}}

\def\sep{{\rm sep}}
\def\bdd{{\rm bdd}}
\def\fB{{\mathfrak B}}
\def\Omega{{\fB}}

\begin{document}

\title{Conjectures about $p$-adic groups\\ and their noncommutative geometry}

\author[A.-M. Aubert]{Anne-Marie Aubert}
\address{Institut de Math\'ematiques de Jussieu -- Paris Rive Gauche, 
U.M.R. 7586 du C.N.R.S., U.P.M.C., 4 place Jussieu 75005 Paris, France}
\email{anne-marie.aubert@imj-prg.fr}
\author[P. Baum]{Paul Baum}
\address{Mathematics Department, Pennsylvania State University, University Park, PA 16802, USA}
\email{pxb6@psu.edu}
\author[R. Plymen]{Roger Plymen}
\address{School of Mathematics, Southampton University, Southampton SO17 1BJ,  England 
\emph{and} School of Mathematics, Manchester University, Manchester M13 9PL, England}
\email{r.j.plymen@soton.ac.uk \quad plymen@manchester.ac.uk}
\author[M. Solleveld]{Maarten Solleveld}
\address{IMAPP, Radboud Universiteit Nijmegen, Heyendaalseweg 135, 
6525AJ Nijmegen, the Netherlands}
\email{m.solleveld@science.ru.nl}
\date{\today}
\subjclass[2010]{20G25, 22E50, 11S37, 19L47}
\maketitle

\begin{abstract}
Let $G$ be any reductive $p$-adic group. We discuss several conjectures, some of them new, 
that involve the representation theory and the geometry of $G$.  

At the heart of these conjectures are statements about the geometric structure of Bernstein 
components for $G$, both at the level of the space of irreducible representations and at the 
level of the associated Hecke algebras. 
We relate this to two well-known conjectures: the local Langlands correspondence and the
Baum--Connes conjecture for $G$. In particular, we present a strategy to reduce the
local Langlands correspondence for irreducible $G$-representations to the local Langlands 
correspondence for supercuspidal representations of Levi subgroups.
\end{abstract}

\vspace{4mm}

\tableofcontents

\section*{Introduction}

This survey paper arose from talks that the first and fourth author gave at the
conference ``Around Langlands correspondences'' in Orsay in June 2015. 
We discuss the representation theory of reductive $p$-adic groups from two different
viewpoints: the Langlands program and noncommutative geometry.
We do this with the aid of several conjectures.

In the first part we formulate a version of the (conjectural) local Langlands
correspondence which is tailored for our purposes. In part 2 we explain what has
become known as the ABPS conjecture. We phrase the most general version, for 
any reductive group over a local nonarchimedean field, not necessarily split.
One of the foundations of this conjecture is the structure of the Hecke algebras
associated to Bernstein components. Based on many known cases we describe in
Conjecture \ref{conj:3} what these algebras should look like in general, up to
Morita equivalence. 

Part 3 focuses on the Galois side of the local Langlands correspondence (LLC).
We conjecture that the space of enhanced L-parameters is in bijection with a
certain union of extended quotients, analogous to the ABPS conjecture. This and
Conjecture \ref{conj:3} have not appeared in print before. Together these conjectures
provide a strategy to reduce the construction of a LLC for a reductive $p$-adic
group to that for supercuspidal representations of its Levi subgroups.

The final part of the paper is purely noncommutative geometric. We discuss the ABPS
conjecture for the topological K-theory of the reduced $C^*$-algebra of a reductive
$p$-adic group. We show that it forms a bridge between the Baum--Connes conjecture
and the LLC.\\

\noindent \textbf{Acknowledgment.}
We thank the referee for several helpful comments.

\section{The local Langlands correspondence}

We briefly discuss the history of the local Langlands correspondence (LLC). With a 
sequence of examplary groups we will reach more and more refined versions of the LLC. 
We will use these examples to explain exactly what kind of L-parameters we want to 
use, and we conjecture a bijective version of this correspondence. 

The (local) Langlands program originated from two sources:
\begin{itemize}
\item (local) class field theory;
\item representation theory of real reductive groups, in particular the work of 
Harish--Chandra on the discrete series.
\end{itemize}
Already in his 1973 preprint \cite{Lan1} Langlands established his correspondence
for real reductive groups: he managed to canonically associate an L-parameter to
every (admissible, smooth) irreducible representation of such a group.

In this paper we focus entirely on the non-archimedean case, so let $F$ be either
a $p$-adic field or a local function field. We fix a separable closure $F_\sep$ and 
we let $\mb W_F \subset \Gal (F_\sep / F)$ be the Weil group of $F$.

\subsection{Tori} \

Let $\mb W_F^{\ab} := \mb W_F / \overline{[\mb W_F,\mb W_F]}$ be the quotient of
$\mb W_F$ by closure of its commutator subgroup.
Recall that Artin reciprocity provides a natural isomorphism of topological groups
\begin{equation}\label{eq:1.1}
\mb a_F : F^\times \to \mb W_F^{\ab} . 
\end{equation}
Langlands had the beautiful idea to interpret this as a statement about $\GL_1 (F)$ 
which admits generalization to other reductive groups. Namely, let $\Irr (F^\times)$ 
be the collection of irreducible smooth complex representations of $F^\times$. 
Of course these are all characters, as $F^\times$ is commutative. Composition with 
\eqref{eq:1.1} gives a bijection
\begin{equation}\label{eq:1.2}
\Hom (\mb W_F,\C^\times) = \Hom (\mb W_F^{\ab},\C^\times) \isom \Irr (F^\times) . 
\end{equation}
(Here and below ``Hom'' means smooth homomorphisms of topological groups.)
More generally, suppose that $S = \cS (F)$ is a $F$-split torus. Let $X^* (S)$
(resp. $X_* (S)$) be the lattice of algebraic characters $\cS \to \GL_1$
(resp. algebraic cocharacters $\GL_1 \to \cS$). These two lattices are canonically
dual to each other and
\[
S \cong X_* (\cS) \otimes_\Z F^\times \cong (F^\times )^{\dim \cS} .
\]
Let $S^\vee := X^* (\cS) \otimes_\Z \C^\times$ be the complex dual torus of 
$S$, characterized by
\[
X^* (S^\vee) = X_* (\cS) ,\quad X_* (S^\vee) = X^* (\cS) .
\]
With Hom-tensor-duality \eqref{eq:1.2} generalizes to
\begin{multline}\label{eq:1.3}
\Irr (S) = \Irr (X_* (\cS) \otimes_\Z F^\times) = \\
\Hom (X_* (\cS) \otimes_\Z F^\times, \C^\times ) \cong
\Hom (F^\times , X^* (\cS) \otimes_\Z F^\times) = \\
\Hom (F^\times, S^\vee) \mosi \Hom (\mb W_F^{\ab},S^\vee) = \Hom (\mb W_F,S^\vee) .
\end{multline}
Motivated by \eqref{eq:1.3}, a Langlands parameter for $S$ is defined to be
a smooth group homomorphism $\mb W_F \to S^\vee$. The collection of such parameters
is denoted $\Phi (S)$, so we can rephrase \eqref{eq:1.3} as a natural bijection
\begin{equation}\label{eq:1.4}
\Irr (S) \to \Phi (S) . 
\end{equation}
Already in 1968 Langlands generalized this to non-split tori.
For example, let $E$ be a finite extension of $F$ contained in $F_\sep$ and let
$T = \Res_{E/F} (E^\times)$, that is, consider $E^\times$ as $F$-group. 
From \eqref{eq:1.4} we get a bijection $\Irr (T) \to \Hom (\mb W_E,\C^\times)$,
and it is desirable to reformulate to right hand side in terms of $\mb W_F$.
Recall that $\mb W_E$ is an open subgroup of $\mb W_F$ of index $[E:F]$. The complex
dual group of $T = \mc T (F)$ is
\[
\mc T^\vee (\C) = T^\vee = \ind_{\mb W_E}^{\mb W_F} (\C^\times) = 
\mathrm{Map}(\mb W_F / \mb W_E,\C^\times) .
\]
It is a complex torus of dimension $[E:F] = \dim_F (T)$ endowed with an action of
$\mb W_F$ via left multiplication on $\mb W_F / \mb W_E$. According to Shapiro's
lemma in continuous group cohomology
\begin{equation}\label{eq:1.5}
\Hom (\mb W_E,\C^\times) = H^1_c (\mb W_E,\C^\times) \cong 
H^1_c (\mb W_F,\ind_{\mb W_E}^{\mb W_F} (\C^\times) = H^1_c (\mb W_F,T^\vee) .
\end{equation}
Langlands \cite{Lan2} showed that the composition of \eqref{eq:1.3} and \eqref{eq:1.5}
is in fact true for every (non-split) torus $T = \mc T (F)$: the group $T^\vee$ is 
always endowed with a canonical action of $\mb W_F$, and there is a natural bijection
\begin{equation}\label{eq:1.6}
\Irr (T) \to H^1_c (\mb W_F , T^\vee) .
\end{equation}
In view of this $H^1_c (\mb W_F,T^\vee)$ is defined to be the space of Langlands
parameters $\Phi (T)$, and \eqref{eq:1.6} is known as the local Langlands correspondence
for tori. More explicitly, $\Phi (T)$ consists of continuous group homomorphisms 
\[
\phi : \mb W_F \to T^\vee \rtimes \mb W_F \quad 
\text{such that} \quad \phi (w) \in T^\vee w \;\; \forall w \in \mb W_F .
\]
Two such homomorphisms $\phi,\phi'$ are considered equal in $H^1_c (\mb W,T^\vee)$ 
if they are conjugate by an element of $T^\vee$, that is, if there is a $t \in T^\vee$ 
such that
\[
\phi' (w) = t \phi (w) t^{-1} \qquad \forall w \in \mb W_F . 
\]

\subsection{Quasi-split groups} \
\label{par:quasi-split}

The most fundamental case of the LLC is the group $\GL_n (F)$. According to Langlands'
original scheme an L-parameter for this group should be an $n$-dimensional representation
$\mb W_F \to \GL_n (\C)$. However, the Bernstein--Zelevinsky classification \cite{Zel}
has shown that not all irreducible representations are obtained in this way. Comparing
$l$-adic and complex representations of $\mb W_F$, Deligne \cite[\S 8]{Del} realized that 
$\mb W_F$ should be replaced by $\mb W_F \ltimes \C$ (now known as the Weil--Deligne 
group). Instead, we use the group $\mb W_F \times \SL_2 (\C)$ as a substitute of the 
Weil--Deligne group (which is possible, as explained in \cite[\S 8]{Kna} and 
\cite[Proposition 2.2]{GrRe}). Thus $\Phi (\GL_n (F))$ is defined as the set of 
isomorphism classes of $n$-dimensional continuous representations 
\[
\phi : \mb W_F \times \SL_2 (\C) \to \GL_n (\C) 
\]
such that $\phi \big|_{\SL_2 (\C)}$ is a homomorphism of algebraic groups.
It was proven in \cite{LRS} that for a local function field $F$ there is a canonical
bijection
\[
\mathrm{rec}_{n,F} : \Irr (\GL_n (F)) \to \Phi (\GL_n (F)) . 
\]
Later this result was also established when $F$ is a $p$-adic field 
\cite{HaTa,Hen,Scho}. We note that all these proofs make use of global methods and
of some very particular Shimura varieties, whose cohomology carries actions of
groups related to $\Gal (F_\sep / F)$ and $\GL_n (F)$. It has turned out to be very
hard to find varieties which play an analogous role for other reductive groups.

What all the above groups have in common, is that the LLC is a canonical 
bijection from $\Irr (G)$ to $\Phi (G)$. This is false for almost any other
group, for example, it already fails for $\SL_2 (F)$. More refinements are needed
to parametrize an L-packet (the set of representations that share the same
L-parameter). We will introduce possible such refinements below.

But first we have to define precisely what we mean by a Langlands parameter for a
general reductive $F$-group $G$. Let $G^\vee = \cG^\vee (\C)$ be the complex dual
group, as in \cite[\S 2]{Bor}. It is endowed with an action of $\Gal (F_\sep / F)$,
in a way which is canonical up to inner automorphisms of $G^\vee$. The group
$G^\vee \rtimes \mb W_F$ is called (the Weil form of) the Langlands dual group 
${}^L G$. Its definition is canonical up to isomorphism. 

From Artin reciprocity we see that Langlands parameters must involve
smooth homomorphisms from the Weil group of $F$, and from the case of split
tori we observe that the target must contain the complex dual group of $G$. 
In fact, the case of non-split tori forces us to take $G^\vee \rtimes 
\mb W_F$ as target and to consider $G^\vee$-conjugacy classes of homomorphisms.
Finally, the case $\GL_n (F)$ shows that we should use $\mb W_F \times \SL_2 (\C)$
as the source of our homomorphisms. 
Through such considerations Borel \cite[\S 8.2]{Bor} arrived at the following notion.

\begin{defn}
A Langlands parameter (or L-parameter for short) $\phi$ for $G$ is smooth group homomorphism
\[
\phi \colon \mb W_F \times \SL_2 (\C) \to G^\vee \rtimes \mb W_F 
\quad \text{ such that:} 
\]
\begin{itemize}
\item $\phi$ preserves the canonical projections to $\mb W_F$, that is,
$\phi (w,x) \in G^\vee w$ for all $w \in \mb W_F$ and $x \in \SL_2 (\C)$;
\item $\phi (w)$ is semisimple for all $w \in \mb W_F$, that is, $\rho (\phi (w,x))$
is semisimple for every finite dimensional representation $\rho$ of 
$G^\vee \rtimes \mb W_F $;
\item $\phi \big|_{\SL_2 (\C)} : \SL_2 (\C) \to G^\vee$ is a homomorphism
of algebraic groups.
\end{itemize}
The group $G^\vee$ acts on the set $\tilde{\Phi} (G)$ of such $\phi$'s by conjugation. 
The set of Langlands parameters for $G$ is defined as the set $\Phi (G)$ of 
$G^\vee$-orbits in $\tilde{\Phi} (G)$.
\end{defn}

We note that $\Phi (G)$ is a subset of $H^1_c (\mb W_F ,G^\vee)$.
The conjectural local Langlands correspondence asserts that there exists a 
canonical, finite-to-one map
\begin{equation}\label{eq:1.7}
\Irr (G) \to \Phi (G) .
\end{equation}
The inverse image of $\phi \in \Phi (G)$ is called the L-packet $\Pi_\phi (G)$. 
Given $\phi \in \tilde{\Phi} (G)$, let $Z_{G^\vee} (\phi)$ be the centralizer of
$\phi (\mb W_F \times \SL_2 (\C))$ in $G^\vee$. Notice that 
\begin{equation}\label{eq:1.20}
Z(G^\vee) \cap Z_{G^\vee}(\phi) = Z(G^\vee)^{\mb W_F}
\end{equation}
by the definition of $\phi$. The (geometric) R-group of $\phi$ is the component group
\begin{equation} \label{eqn:geoR-group}
\cR_\phi := \pi_0 \big( Z_{G^\vee}(\phi) / Z(G^\vee )^{\mb W_F} \big) .
\end{equation}
It is clear that, up to isomorphism, $\cR_\phi$ depends only on the image of $\phi$ in
$\Phi (G)$.

Suppose now that $G$ is quasi-split over $F$. Then it is expected that $\Pi_\phi (G)$ is
in bijection with $\Irr (\mf R_\phi)$. This was first suggested in a special case in
\cite[\S 1.5]{Lus1}. When $F$ is $p$-adic this was proven for 
quasi-split orthogonal and symplectic groups in \cite{Art2}, for corresponding 
quasi-split similitude groups in \cite{Xu}, and for quasi-split unitary 
groups in \cite{Mok}. The main method in these works is twisted endoscopic transfer,
they rely on the LLC for $\GL_n (F)$.

\subsection{Inner forms and inner twists} \

General connected reductive $F$-groups need not be quasi-split, but they are always forms
of split $F$-groups. Let us recall 
the parametrization of forms by means of Galois cohomology.
Two $F$-groups $G = \cG (F)$ and $G_2 = \cG_2 (F)$ are called forms of each other if
$\cG$ is isomorphic to $\cG_2$ as algebraic groups, or equivalently if
$\cG (F_\sep) \cong \cG_2 (F_\sep)$ as $F_\sep$-groups. An isomorphism $\alpha:
\cG_2 \to \cG$ determines a 1-cocycle
\begin{equation}\label{eq:1.10}
\gamma_\alpha : \begin{array}{ccc} \Gal (F_\sep / F) & \to & \Aut (\cG) \\
\sigma & \mapsto & \alpha \sigma \alpha^{-1} \sigma^{-1} .                 
\end{array}
\end{equation}
From $\gamma_\alpha$ one can recover $G_2$ (up to isomorphism) as
\[
G_2 \cong \{ g \in \cG (F_\sep) : (\gamma_\alpha (\sigma) \circ \sigma) g = g 
\quad \forall \sigma \in \Gal (F_\sep / F) \} .
\]
Given another form $\beta : \cG_3 \to \cG$, the groups $G_2$ and $G_3$ are
$F$-isomorphic if and only if the 1-cocycles $\gamma_\alpha$ and $\gamma_\beta$
are cohomologous. That is, if there exists a $f \in \Aut (\cG)$ such that
\begin{equation} \label{eq:cohomologous}
\gamma_\alpha (\sigma) = f^{-1} \gamma_\beta (\sigma) \; \sigma f \sigma^{-1}
\quad \forall \sigma \in \Gal (F_\sep / F) .
\end{equation}
In this way the isomorphism classes of forms of $G = \cG (F)$ are in bijection with
the Galois cohomology group $H^1 (F,\Aut (\cG))$. By definition $G_2$ is an
inner form of $G$ if the cocycle $\gamma_\alpha$ takes values in the group of
inner automorphisms $\mathrm{Inn}(\cG)$ (which is isomorphic to the adjoint group
$\cG_{\ad}$). On the other hand, if the values of $\gamma_\alpha$ are not contained 
in Inn$(\cG)$, then $G_2$ is called an outer form of $G$.

By \cite[\S 16.4]{Spr} that every connected reductive $F$-group is an 
inner form of a unique quasi-split $F$-group. It is believed that in the Langlands
program it is advantageous to study all inner forms of a given group
simultaneously. One reason is that the inner forms share the same Langlands dual 
group, because the action of $\mb W_F$ on $G^\vee$ is only uniquely defined up to
inner automorphisms. Hence two inner forms have the same set of Langlands parameters.
This also works the other way round: from the Langlands dual group ${}^L G$ one can
recover the inner form class of $G$.

Later we will see that it is even better to 
consider not inner forms, but rather inner twists of a fixed (quasi-split) group.
An inner twist consists of a pair $(G_2,\alpha)$ as above, where $G_2 = \cG_2 (F)$
and $\alpha : \cG_2 \isom \cG$ are such that im$(\gamma_\alpha) \subset \cG_{\ad}$.
Two inner twists of $G$ are equivalent if \eqref{eq:cohomologous} holds for
some $f \in \mathrm{Inn}(\cG)$. 
The equivalence classes of inner twists of $G$ are parametrized by the Galois
cohomology group $H^1 (F,\cG_{\ad})$. 

It is quite possible that two inequivalent inner twists $(G_2,\alpha)$ and
$(G_3,\beta)$ share the same group $G_2 \cong G_3$. This happens precisely
when $\gamma_\alpha$ and $\gamma_\beta$ are in the same orbit of
$\Aut (\cG) / \mathrm{Inn}(\cG)$ on $H^1 (F,\cG_{\ad})$.

Kottwitz has found an important alternative description of $H^1 (F,\cG)$.
Recall that the complex dual group $G^\vee = \cG^\vee (\C)$ is endowed with
an action of $\Gal (F_\sep / F)$. 

\begin{prop}\label{prop:Kot} \textup{\cite[Proposition 6.4]{Kot}} \ \\
There exists a natural isomorphism
\[
\kappa_G : H^1 (F,\cG) \isom \Irr \Big( \pi_0 \big( Z (G^\vee)^{\mb W_F} \big) \Big) 
\]
\end{prop}

This is particularly useful in the following way. An inner twist of $G$ is the same
thing as an inner twist of the unique quasi-split inner form $G^* = \cG^* (F)$.
Let $G^*_{\ad} = \cG^*_{\ad}(F)$ be the adjoint group of $G^*$ and
let $G^\vee_{\sc} = (G_{\ad})^\vee$ be the simply connected cover of the derived
group of $G^\vee$. Here the Kottwitz isomorphism becomes
\begin{equation}\label{eq:1.8}
\kappa_{G^*_{\ad}} \colon H^1 (F,\cG^*_{\ad}) \isom \Irr \Big( Z (G^\vee_{\sc})^{\mb W_F} \Big) .
\end{equation}
This provides a convenient way to parametrize inner twists of $G$.\vspace{3mm}

\textbf{Example.} 
We work out the above when $\cG = \GL_n$, relying on \cite[\S XI.4]{Wei}.
From \eqref{eq:1.8} we see that
\[
H^1 (F,\cG_{\ad}) \cong \Irr \big( Z(\SL_n (\C)) \big) 
\]
is cyclic of order $n$. Let $F_{(n)}$ be the unique unramified extension of $F$ of 
degree $n$, and let $\Fr \in \mb W_F$ be an arithmetic Frobenius element. Thus
\[
\Gal (F_\sep / F) / \Gal (F_{(n)} / F) \cong \Gal (F_{(n)} / F) \cong
\langle \Fr \rangle / \langle \Fr^n \rangle \cong \Z / n \Z .
\]
Let $\varpi_F$ be a uniformizer and define $\gamma \in H^1 (F,\PGL_n)$ by
\[
\gamma (\Fr^m ) = \theta_n^m ,\quad \theta_n = 
\begin{pmatrix}
0 & & & \varpi_F \\
1 & 0 \\
 & \ddots & \ddots \\
 & & 1  & 0
\end{pmatrix} .
\]
Then $\gamma$ generates $H^1 (F,\PGL_n)$. One can check that 
\[
D_\gamma := \{ A \in M_n (F_{(n)}) : \theta_n \Fr (A) \theta_n^{-1} = A \}
\]
is generated by $\theta_n$ and the matrices diag$(a,\Fr (a),\ldots,\Fr^{n-1}(a))$
with $a \in F_{(n)}$. It is a division algebra of dimension $n^2$ over its centre $F$,
called the cyclic algebra $[F_{(n)} / F,\chi,\varpi_F]$ by Weil. The associated
inner twist of $\GL_n (F)$ is 
\[
\Big(\GL_1 (D_\gamma), \text{inclusion } D_\gamma^\times \to \GL_n (F_{(n)}) \Big)
\]
For $m < n ,\; \gamma^m \in H^1 (F,\PGL_n)$ is of order $d = n / \gcd (n,m)$.
One obtains a division algebra $D_{\gamma^m}$ of dimension $d^2$ over $F$, 
contained in $M_d (F_{(d)})$. The associated inner twist is
\[
\Big( \GL_{\gcd (n,m)}(D_{\gamma^m}), \text{ inclusion in } \GL_n (F_{(d)}) \Big) .
\]
In this way one finds that the inner twists of $\GL_n (F)$ are in bijection with
the isomorphism classes of division algebras with centre $F$,
whose dimension divides $n^2$.

Two inner forms $\GL_m (D)$ and $\GL_{m'}(D')$ can be isomorphic even when $D$ is
not isomorphic to $D'$. For example let $D^{op}$ be the opposite algebra of $D$
and denote the inverse transpose of a matrix $A$ by $A^{-T}$. Then
\begin{equation}\label{eq:1.9}
\GL_m (D) \to \GL_m (D^{op}) : A \mapsto A^{-T}
\end{equation}
is a group isomorphism. The group $\Aut (\GL_n) / \mathrm{Inn}(\GL_n)$ has order two,
the nontrivial element is represented by the inverse transpose map $-T$. The 
isomorphism \eqref{eq:1.9} reflects the action of $\Aut (\GL_n) / \mathrm{Inn}(\GL_n)$
on $H^1 (F,\PGL_n) \cong \Z / n \Z$ by $-T \cdot m = -m$. Hence the isomorphism
classes of inner forms of $\GL_n (F)$ are bijection with
\[
H^1 (F,\mathrm{Inn}(\GL_n)) / \{ 1, -T\} \cong (\Z / n \Z) / \{ \pm 1\} .
\]
All the outer forms of $\GL_n (F)$ are unitary groups.
If $\gamma : \Gal (F_\sep / F) \to \Aut (\GL_n)$ is a non-inner 1-cocycle, then
\[
\ker \Big( \Gal (F_\sep / F) \xrightarrow{\gamma} \Aut (\GL_n) / \mathrm{Inn}(\GL_n) \Big)
\]
is an index two subgroup of $\Gal (F_\sep / F)$. It defines a separable quadratic
field extension $E / F$. One such cocycle is given by
\[
\gamma (\sigma) = \left\{\begin{array}{cl}
\mathrm{id} & \sigma \in \Gal (F_\sep / E) , \\
\mathrm{Ad}(J_n) \circ -T & \sigma \in \Gal (F_\sep / F) \setminus \Gal (F_\sep / E)
\end{array} \right. \quad
J_n = \begin{pmatrix}
0 & \cdots & 0 & 1 \\
\vdots &\text{\reflectbox{$\ddots$}} & 1 & 0 \\
0 & \text{\reflectbox{$\ddots$}} & \text{\reflectbox{$\ddots$}}& \vdots \\
1 & 0 & \cdots & 0 
\end{pmatrix}
\]
The corresponding outer form $\GL_{n,\gamma}$ is
\[
\rU_n (E/F) := \{ A \in \GL_n (E) : (\gamma (\sigma) \circ \sigma) A = A \quad
\forall \sigma \in \Gal (F_\sep / F) \} .
\]
This is the unitary group associated with the Hermitian form on $E^n$ determined by
$J_n$ and $F$. It is quasi-split, the upper triangular matrices in $\rU_n (E/F)$
form a Borel subgroup. From \eqref{eq:1.8} we see that 
\[
\big| H^1 (F,\GL_{n,\gamma}) \big| = \big| Z(\SL_n (\C))^{\mb W_F} \big| ,
\]
where $\mb W_F$ acts on $\SL_n (\C)$ via $\gamma$. Hence
\[
H^1 (F,\GL_{n,\gamma}) = \left\{ \begin{array}{cl}
\Z / 2 \Z & n \text{ even} \\
1 & n \text{ odd.}
\end{array}
\right.
\]
When $n$ is even, the unique other inner form of $U_n (E/F)$ is a unitary group
associated to another $n$-dimensional Hermitian space over $E/F$. It can be constructed
as above, but with a matrix $J_{n-2} \oplus \matje{1}{0}{0}{a}$ instead of $J_n$.
Here $\matje{1}{0}{0}{a}$ represents a two-dimensional anisotropic Hermitian space.

\subsection{Enhanced L-parameters and relevance} \

In spite of the successes for quasi-split classical groups, for more general groups, 
the R-group $\cR_\phi$ cannot always parametrize the L-packet $\Pi_\phi (G)$, 
this was already noticed in \cite{Art1}. In fact, $\Pi_\phi (G)$
can very well be empty if $G$ is not quasi-split. 

To overcome this problem, the notion of relevance of L-parameters was devised. 
It is derived from relevance of parabolic and Levi subgroups. (Below and later, we 
call a Levi factor of a parabolic subgroup of $G$ simply a Levi subgroup of $G$.)
Let $\cT$ be a maximal torus of $\cG$ and let $\Delta$ be a basis of the root
system $R(\cG,\cT)$. Recall \cite[Theorem 8.4.3]{Spr} that the set of conjugacy 
classes of parabolic subgroups of $\cG$ is in bijection with the power set of $\Delta$.
The bijection 
\[
R(\cG,\cT) \longleftrightarrow R^\vee (\cG,\cT) = R(G^\vee,T^\vee)
\]
gives a basis $\Delta^\vee$, and provides a canonical bijection between the sets of
conjugacy classes of parabolic subgroups of $\cG$ and of $G^\vee$.

As in \cite[\S 3]{Bor}, we say that a parabolic subgroup $P^\vee$ of $G^\vee$ is
$F$-relevant if the corresponding class of parabolic subgroups of $\cG$ contains
an element $\mc P$ which is defined over $F$. Similarly, we call a Levi subgroup
$M^\vee \subset G^\vee \; F$-relevant if it is a Levi factor of a parabolic subgroup
$P^\vee \subset G^\vee$ which is $F$-relevant.

We say that a parabolic subgroup $P^\vee$ of $G^\vee$ is quasi-stable under $\mb W_F$ 
if the projection $N_{G^\vee \rtimes \mb W_F} (P^\vee) \to \mb W_F$ is surjective. 
These are precisely the neutral components of what Borel \cite[\S 3]{Bor} calls 
parabolic subgroups of $G^\vee \rtimes \mb W_F$.

\begin{defn}\label{def:relevance}
Let $\phi \in \tilde{\Phi} (G)$ and let $P^\vee$ be a $\mb W_F$-quasi-stable parabolic 
subgroup of $G^\vee$ with a Levi factor $M^\vee$ such that
\begin{itemize}
\item the image of $\phi$ is contained in $N_{P^\vee \rtimes \mb W_F}(M^\vee)$;
\item $P^\vee$ is a minimal for this property.
\end{itemize}
Then $\phi$ is called relevant for $G$ if $P^\vee$ is $F$-relevant.
\end{defn}

We remark that above one cannot substitute the first requirement by
``the image of $\phi$ is contained in $P^\vee \rtimes \mb W_F$'', that would
give an unsatisfactory notion of relevance. It is expected that in general 
$\Pi_\phi (G)$ is nonempty if and only if $\phi$ is relevant for $G$.\\

\textbf{Example.} Let $G = D^\times$ be the multiplicative group of a 4-dimensional
noncommutative division algebra over $F$. It is the unique non-split inner
form of $\GL_2 (F)$. The only Levi subgroup of $D^\times$ defined over $F$ is
$D^\times$ itself, and it corresponds to the Levi subgroup $\GL_2 (\C)$ on the
complex side.

Consider $\phi_1 \in \tilde{\Phi} (\GL_2 (F)) = \tilde{\Phi} (D^\times)$
which is just the embedding $\mb W_F \times \SL_2 (\C) \to \GL_2 (\C) \times \mb W_F$.
No proper parabolic subgroup of $\GL_2 (\C)$ contains $\phi_1 (\SL_2 (\C)) = \SL_2 (\C)$,
so $\phi_1$ is relevant for both $D^\times$ and $\GL_2 (F)$. Indeed, 
$\Pi_\phi (\GL_2 (F))$ is the Steinberg representation of $\GL_2 (F)$ and 
$\Pi_\phi (D^\times)$ is the Steinberg representation of $D^\times$ (which is
just the trivial representation).

On the other hand, suppose that $\phi_2 \in \tilde{\Phi} (\GL_2 (\C))$ with 
\[
\phi_2 (\SL_2 (\C)) = 1 \quad \text{and} \quad 
\phi_2 (\mb W_F) \subset \text{diag}(\GL_2 (\C)) \times \mb W_F .
\]
Then $M^\vee = \text{diag}(\GL_2 (\C))$ is the minimal Levi subgroup such that
$M^\vee \times \mb W_F$ contains the image of $\phi_2$. Thus the standard Borel
subgroup $P^\vee$ of $\GL_2 (\C)$ satisfies the conditions in Definition 
\ref{def:relevance}. But its conjugacy class does not correspond to any parabolic 
subgroup of $D^\times$, so $\phi_2$ is not relevant for $D^\times$.

\vspace{2mm}
To parametrize L-packets, we must add some extra data to our Langlands parameters $\phi$.
In view of the quasi-split case we need at least the irreducible representations of
the geometric R-group $\cR_\phi$, but that is not enough. We will use enhancements that 
carry information about both the R-group of $\phi$ and the inner twists of $G$. 
We will follow Arthur's set-up in \cite{Art1}.

Recall that $G^\vee_{\sc}$ is the simply connected cover of both the derived group 
$G^\vee_{\der}$ and the adjoint group $G^\vee_{\ad}$ of $G^\vee$. It acts on 
$\tilde{\Phi} (G)$ by conjugation, via the natural map $G^\vee_{\sc} 
\to G^\vee_{\der}$. For $\phi \in \tilde{\Phi} (G)$, let $Z_{G^\vee_{\sc}}(\phi)$ 
be the centralizer of $\phi (\mb W_F \times \SL_2 (\C))$ in $G^\vee_{\sc}$.
By \eqref{eq:1.20}
\begin{equation}\label{eq:1.21}
Z_{G^\vee}(\phi) / Z (G^\vee)^{\mb W_F} \cong Z_{G^\vee}(\phi) Z(G^\vee) / Z(G^\vee) .
\end{equation}
We can regard the right hand side as a subgroup of $G^\vee_\ad$. Let 
$Z^1_{G^\vee_\sc}(\phi)$ be its inverse under the projection $G^\vee_\sc \to G^\vee_\ad$. 
Although $Z^1_{G^\vee_\sc}(\phi)$ contains $Z_{G^\vee_\sc}(\phi)$ as a 
normal subgroup of finite index, not all its elements fix $\phi$. More precisely
\[
Z^1_{G^\vee_\sc}(\phi) = \{ g \in G^\vee_\sc : g \phi g^{-1} = \phi \, a_g
\text{ for some } a_g \in B^1 (\mb W_F ,Z(G^\vee)) \} .
\]
Here $B^1 (\mb W_F,Z(G^\vee))$ is the set of 1-coboundaries for group cohomology, 
that is, maps $\mb W_F \to Z(G^\vee)$ of the form $w \mapsto z w z^{-1} w^{-1}$ 
with $z \in Z(G^\vee)$.

The difference between $Z_{G^\vee_\sc}(\phi)$ and $Z^1_{G^\vee_\sc}(\phi)$ is caused 
by the identification \eqref{eq:1.21}, which as it were includes $Z(G^\vee)$ in 
$Z_{G^\vee}(\phi)$. We note that $Z^1_{G^\vee_\sc}(\phi) = Z_{G^\vee_\sc}(\phi)$
whenever $Z(G^\vee_\sc)^{\mb W_F} = Z(G^\vee_\sc)$, in particular if $G$ is an inner
twist of a split group. On the other hand, if $Z(G^\vee_\sc)^{\mb W_F} \neq Z(G^\vee_\sc)$,
then it does not suffice to consider $Z_{G^\vee_\sc}(\phi)$, that would not necessarily
account for all elements of $Z_{G^\vee}(\phi)$.

\begin{defn}
The S-group of $\phi$ is the component group 
$\cS_\phi = \pi_0 \Big( Z^1_{G^\vee_{\sc}}(\phi) \Big)$.

An enhancement of $\phi$ is an irreducible complex representation of $\cS_\phi$.
\end{defn} 

The next lemma implies that every irreducible representation of $\cR_\phi$ lifts to
one of $\cS_\phi$.

\begin{lem}\label{lem:1.1}
Write $\cZ_\phi = Z(G^\vee_{\sc}) \big/ \big( Z(G^\vee_{\sc}) \cap 
Z_{G^\vee_{\sc}}(\phi)^\circ \big)$. These groups fit in a natural central extension
\[
1 \to \cZ_\phi \to \cS_\phi \to \cR_\phi \to 1.
\]
\end{lem}
\begin{proof}
First we note that $Z(G^\vee_{\sc})$ is contained in the centre of 
$Z^1_{G^\vee_{\sc}}(\phi)$. As $Z^1_{G^\vee_{\sc}}(\phi)^\circ = Z_{G^\vee_{\sc}}
(\phi)^\circ$, this means that $\cZ_\phi$ is a central subgroup of $\cS_\phi$.

The kernel of the natural map $Z^1_{G^\vee_{\sc}}(\phi) \to Z_{G^\vee}(\phi) Z(G^\vee) 
/ Z(G^\vee)$ is $Z(G^\vee_\sc)$, so $\cZ_\phi = \ker (\cS_\phi \to \cR_\phi)$.

Consider any $g \in Z_{G^\vee}(\phi)$. Pick $g_1 \in G^\vee_\der$ and 
$g_2 \in Z(G^\vee)^\circ$ so that $g_2 g_1 = g$. For any 
$(w,x) \in \mb W_F \times \SL_2 (\C)$ we have $\phi (w,x) \in G^\vee w$ and
\begin{multline*}
g^{-1}_1 g^{-1}_2 = g^{-1} = \phi (w,x) g^{-1} \phi (w,x) = 
\phi (w,x) g^{-1}_1 \phi (w,x)^{-1} \phi (w,x) g^{-1}_2 \phi (w,x)^{-1} \\
= \phi (w,x) g_1^{-1} \phi (w,x)^{-1} w g_2^{-1} w^{-1} .
\end{multline*}
Hence $g_1 \phi (w,x) g_1^{-1} \phi (w,x)^{-1} = g_2^{-1} w g_2 w^{-1} \in 
Z(G^\vee)^\circ \cap G^\vee_\der$. In other words,
\[
g_1 \phi (w,x) g_1^{-1} = \phi (w,x) a (w) \text{ where } a(w) = g_2^{-1} w g_2 w^{-1} .
\]
Let $g_3 \in G^\vee_\sc$ be a lift of $g_1 \in G^\vee_\der$. Then also $g_3 \phi g_3^{-1}
= \phi \, a$, showing that $g_3 \in Z^1_{G^\vee_\sc}(\phi)$. The image of $g_3$ in
$Z_{G^\vee}(\phi) Z(G^\vee) / Z(G^\vee)$ is $g_1 Z(G^\vee) = g_2 g_1 Z(G^\vee) =
g Z(G^\vee)$. Thus $\cS_\phi \to \cR_\phi$ is surjective.  
\end{proof}

Let us write $\cZ_\phi^{\mb W_F} = Z(G^\vee_\sc)^{\mb W_F} \big/ \big( Z(G^\vee_\sc
)^{\mb W_F} \cap Z_{G^\vee_\sc}(\phi)^\circ \big)$. According to \cite[\S 4]{Art1}
\begin{equation}\label{eq:1.22}
Z(G^\vee_\sc) \cap Z_{G^\vee_\sc}(\phi)^\circ \subset Z(G^\vee_\sc)^{\mb W_F} .
\end{equation}
Hence $\cZ_\phi^{\mb W_F}$ can be regarded as a subgroup of $\cZ_\phi$ and
\begin{equation}\label{eq:1.23}
\cZ_\phi / \cZ_\phi^{\mb W_F} \cong Z(G^\vee_\sc) / Z(G^\vee_\sc)^{\mb W_F} .
\end{equation}
By Schur's lemma every enhanced Langlands parameter $(\phi,\rho)$ restricts to a character
$\rho |_{\cZ_\phi^{\mb W_F}}$ of $\cZ_\phi^{\mb W_F}$. This can be inflated to a character 
$\zeta_\rho$ of $Z(G^\vee_{\sc})^{\mb W_F}$. With the Kottwitz isomorphism \eqref{eq:1.8} 
we get an element $\kappa_{G^*_{\ad}}^{-1}(\zeta_\rho) \in H^1 (F,\cG^*_{\ad})$. In this way 
$(\phi,\rho)$ determines a unique inner twist of $G$. This can be regarded as an alternative
way to specify for which inner twists of $G$ an enhanced Langlands parameter is relevant.
Fortunately, it turns out that it agrees with the earlier definition of relevance of
Langlands parameters.

\begin{prop}\label{prop:1.2}
Let $\zeta \in \Irr (Z (G^\vee_{\sc})^{\mb W_F})$ and let $G_\gamma$ be the inner 
twist of $G$ associated to $\gamma = \kappa_{G^*_{\ad}}^{-1}(\zeta)$ via \eqref{eq:1.8}. 
For $\phi \in \tilde{\Phi} (G)$ the following are equivalent:
\begin{enumerate}
\item $\phi$ is relevant for $G_\gamma$;
\item $Z(G^\vee_{\sc})^{\mb W_F} \cap Z_{G^\vee_{\sc}}(\phi)^\circ \subset \ker \zeta$;
\item there exists a $\rho \in \Irr (\cS_\phi)$ such that $\zeta$ is the lift of
$\rho \big|_{\cZ_\phi^{\mb W_F}}$ to $Z(G^\vee_{\sc})^{\mb W_F}$.
\end{enumerate}
\end{prop}
\begin{proof}
(1) $\Longleftrightarrow$ (2) See \cite[Lemma 9.1]{HiSa} and \cite[Corollary 2.3]{Art0}.
We note that what Hiraga and Saito call $S_\phi^\circ$ equals $Z_{G^\vee_{\sc}}(\phi)^\circ$.\\
(2) $\Longrightarrow$ (3) Obvious.\\
(2) $\Longleftarrow$ (3) The assumption says that $\zeta$ can be regarded as a character
of $\cZ_\phi^{\mb W_F}$. The induced $\cS_\phi$-representation 
$\ind_{\cZ_\phi^{\mb W_F}}^{\cS_\phi} (\zeta)$ has finite dimension and 
$Z(G^\vee_{\sc})^{\mb W_F}$ acts on it as $\zeta$. Let $\rho$ be
any irreducible constituent of $\ind_{\cZ_\phi^{\mb W_F}}^{\cS_\phi} (\zeta)$.
\end{proof}

Supported by the above result, we extend the definition of relevance to inner twists and
enhanced L-parameters.

\begin{defn}
Let $(G,\alpha)$ be an inner twist of a quasi-split $F$-group $G^*$. 
Let $\phi \in \tilde{\Phi} (G^*) = \tilde{\Phi} (G)$ and let $\rho \in \Irr (\cS_\phi)$. 
We call $\rho$ relevant for $(G,\alpha)$ if
\[
\kappa_{G^*_{\ad}}^{-1}(\zeta_\rho) = \gamma_\alpha,
\]
where $\zeta_\rho = \rho \big|_{\cZ_\phi^{\mb W_F}}$ and $\gamma_\alpha \in H^1(F,\cG^*_{\ad})$ 
is defined in \eqref{eq:1.10}.

We denote the space of such relevant pairs $(\phi,\rho)$ by $\Phi_e (G)$. 
The group $G^\vee_{\sc}$ acts on $\tilde{\Phi}_e (G)$ by 
\[
g \cdot (\phi,\rho) = (g \phi g^{-1},g \cdot \rho), \text{ where }
(g \cdot \rho)(g h g^{-1}) = \rho (h) \text{ for } h \in Z_{G^\vee_{\sc}}(\phi) .
\]
A $G^\vee_{\sc}$-orbit in $\tilde{\Phi}_e (G)$ is called an enhanced L-parameter for $G$, 
and the set of those is denoted $\Phi_e (G)$.
\end{defn}

\subsection{A bijective version of the LLC} \
\label{par:bijective}

We are ready to formulate our version of the conjectural local Langlands correspondence.
It is inspired by many sources, in particular \cite[\S 10]{Bor}, \cite[\S 4]{Vog}, 
\cite[\S 3]{Art1} and \cite[\S 5.2]{Hai}.

In some cases $\cS_\phi$ is too large, because we have included the entire group $Z(G^\vee_\sc)$.
To compensate for this it is handy to restrict our enhancements of L-parameters to a
subset of $\Irr (\cS_\phi)$.  By Lemma \ref{lem:1.1} and  Schur's lemma, the enhancement 
$\rho$ restricts to a character of $\cZ_{\phi}^{\mb W_F}$, 
which then inflates to a character $\zeta_{\rho}$ of $Z(G^{\vee}_{\sc})^{\mb W_F}$. If $\rho$
is relevant for $G$, then $\zeta_\rho = \kappa_{G^*_{\ad}}(\gamma_\alpha) \in \Irr 
\big( Z(G^\vee_\sc)^{\mb W_F})$. It can be extended in precisely 
$[Z(G^\vee_\sc) : Z(G^\vee_\sc)^{\mb W_F}]$ ways to a character of $Z(G^\vee_\sc)$. We choose 
such an extension and we denote it by $\zeta_G$. By Proposition \ref{prop:1.2} every 
$\phi \in \tilde{\Phi}(G)$ can be enhanced with a $\rho \in \Irr (\cS_\phi)$
such that $\rho \big|_{\cZ_\phi}$ inflates to $\zeta_G$. 

We denote the set of equivalence classes of such $(\phi,\rho) \in \tilde{\Phi}_e (G)$ by 
$\Phi_{e,\zeta_G}(G)$. Of course we pick $\zeta_G =$ triv when $G$ is quasi-split. In that
case Lemma \ref{lem:1.1} shows that $\Phi_{e,\mathrm{triv}}(G)$ agrees with the set of enhanced
L-parameters for $G$ discussed in Paragraph \ref{par:quasi-split}.

\begin{conj}\label{conj:1}
Let $(G,\alpha)$ be an inner twist of a quasi-split $F$-group.
There exists a surjection
\[
\Phi_e (G) \longrightarrow \Irr (G) : (\phi,\rho) \mapsto \pi_{\phi,\rho} ,
\]
which becomes bijective when restricted to $\Phi_{e,\zeta_G}(G)$. We write its inverse as
\[
\Irr (G) \longrightarrow \Phi_{e,\zeta_G}(G) : \pi \mapsto (\phi_\pi,\rho_\pi) .
\]
Then the composed map $\Irr (G) \to \Phi (G) : \pi \mapsto \phi_\pi$ is canonical.
These maps satisfy the properties (1) -- (7) listed below.
\end{conj}

We remark that the above bijection becomes more elegant if one 
considers the union over inner twists, then it says that there exists a surjection
\[
\{ (\phi,\rho) : \phi \in \Phi (G^*), \rho \in \Irr (\cS_\phi) \} \to
\{ (G,\alpha,\pi) : (G,\alpha) \text{ inner twist of } G^*, \pi \in \Irr (G) \} 
\]
whose fibers have exactly $[Z(G^\vee_\sc) : Z(G^\vee_\sc)^{\mb W_F}]$ elements.

Before we write down the additional properties, we recall two notions for L-pa\-ra\-me\-ters.
Let $\phi \in \tilde{\Phi} (G)$. We say that $\phi$ is discrete (or elliptic) if there is no
proper $\mb W_F$-stable Levi subgroup $M^\vee \subset G^\vee$ such that $\phi (\mb W_F
\times \SL_2 (\C)) \subset M^\vee \rtimes \mb W_F$. We call $\phi$ bounded if 
$\phi' (\mb W_F) \subset G^\vee$ is bounded, where $\phi (w) = (\phi' (w),w)$. (This
is equivalent to $\phi' (\Fr)$ being a compact element of $G^\vee$.)\\

\textbf{Desiderata for the local Langlands correspondence (Borel).}
\begin{enumerate}
\item The central character of $\pi$ equals the character of $Z(G)$ constructed from
$\phi_\pi$ in \cite[\S 10.1]{Bor}.
\item Let $z \in H^1_c (\mb W_F, Z (G^\vee))$ be a class in continuous group cohomology, and
let $\chi_z : G \to \C^\times$ be the character associated to it in \cite[\S 10.2]{Bor}.
Thus $z \phi_\pi \in \tilde{\Phi} (G)$ and $\cS_{z \phi_\pi} = \cS_{\phi_\pi}$. 
Then the LLC should satisfy $(z \phi_\pi,\rho_\pi) = (\phi_{\chi_z \pi} , \rho_{\chi_z \pi})$. 
\item $\pi$ is essentially square-integrable if and only if $\phi_\pi$ is discrete.
\item $\pi$ is tempered if and only if $\phi_\pi$ is bounded.
\item Let $P$ be a parabolic subgroup of $G$ with Levi factor $M$.
Suppose that $g \in N_G (M)$ and $\check g \in N_{G^\vee}(M^\vee)$ are such that
Ad$(g) : M \to M$ and Ad$(\check g) : M^\vee \to M^\vee$ form a corresponding pair of
homomorphisms, in the sense of \cite[\S 2]{Bor}. Then 
\[
(\phi_{g \cdot \pi}, \rho_{g \cdot \pi}) = (\mathrm{Ad}(\check g) \phi_\pi, 
\check g \cdot \rho_\pi) \quad \text{for all} \quad \pi \in \Irr (M).
\]
\item Suppose that $(\phi^M,\rho^M) \in \Phi_e (M)$ is bounded. Then
\begin{equation}\label{eq:1.11}
\{ \pi_{\phi,\rho} : \phi = \phi^M \text{ composed with } {}^L M \to {}^L G,
\rho \big|_{\cS_\phi^M} \text{ contains } \rho^M \}
\end{equation}
equals the set of irreducible constituents of the parabolically induced representation
$I_P^G (\pi_{\phi^M ,\rho^M})$.
\item If $\phi^M$ is discrete but not necessarily bounded, then \eqref{eq:1.11} is the set of
Langlands constituents of $I_P^G (\pi_{\phi^M ,\rho^M})$, as in \cite[p. 30]{ABPS1}.
\end{enumerate}

We note that in order to establish Conjecture \ref{conj:1} for (a collection of) groups,
it suffices to prove it for tempered representations and bounded enhanced L-parameters.
This follows from comparing the geometry of the spaces $\Irr (G)$ and $\Phi_e (G)$
\cite{ABPS2}, or from the Langlands classification for $\Irr (G)$ \cite[\S VII.4]{Ren} 
and its counterpart for L-parameters \cite{SiZi}.

Of course one can hope for many more properties, like compatibility with L-functions, 
adjoint $\gamma$-factors \cite{HII} and functoriality. For our survey (1)--(7) are
sufficient. This bijective version of the LLC, including the listed properties, is known
in the following cases:
\begin{itemize}
\item General linear groups over division algebras, or more precisely inner twists of 
$\GL_n (F)$. It is a consequence of the LLC for $\GL_n (F)$ and the Jacquet--Langlands 
correspondence \cite{DKV,Bad}, see \cite[Theorem 2.2]{ABPS3}.
\item Inner twists of $\SL_n (F)$, see \cite[\S 12]{HiSa} and \cite[Theorem 3.3]{ABPS3}.
\item Orthogonal and symplectic groups \cite{Art2} and similitude groups \cite{Xu}.
\item Unitary groups \cite{Mok,KMSW}.
\item Principal series representations of split groups \cite[\S 16]{ABPS5}.
\item Unipotent representations of adjoint groups \cite{Lus6}.
\item Epipelagic representations of tamely ramified groups \cite{Kal}.
\end{itemize}
The last three items concern particular classes of representations of certain groups. 
All the groups for which the complete LLC is currently known are linked to $\GL_n (F)$, 
and the proofs for these groups use the LLC for general linear groups in an essential way. 
It appears to be a big challenge to find an approach to the LLC which does not rely on 
the case of $\GL_n (F)$, and can be applied to more general reductive groups.

\section{The smooth dual of a reductive $p$-adic group}

Let $G$ be a connected reductive group over a local non-archimedean field, and
let $\Irr (G)$ be the set of irreducible (smooth, complex) $G$-representations. In this
section we discuss the geometric structure of $\Irr (G)$. It is topologized via the
Jacobson topology for the Hecke algebra of $G$, and in this way it is automatically
rather close to an algebraic variety.
We propose a generalization of our earlier conjectures \cite{ABP,ABPS2}, which
make the structure of $\Irr (G)$ much more precise. To formulate these conjectures, we need
extended quotients and the Bernstein decomposition.

\subsection{Twisted extended quotients} \ \label{par:extquot}

Let $\Gamma$ be a group acting on a topological space $X$. In 
\cite[\S 2]{ABPS5} we studied various extended quotients of $X$ by $\Gamma$. In this 
paper we need the most general version, the twisted extended quotients.

Let $\natural$ be a given function which assigns to each $x \in X$ a 2-cocycle 
\[
\natural_x : \Gamma_x \times \Gamma_x \to \C^\times \text{, where } 
\Gamma_x = \{\gamma \in \Gamma : \gamma x = x\}. 
\]
Recall that the twisted group algebra $\C [\Gamma_x,\natural_x]$ has a basis
$\{ N_\gamma : \gamma \in \Gamma_x \}$ and multiplication rules
\begin{equation}\label{eq:2.20}
N_\gamma N_{\gamma'} = \natural_x (\gamma,\gamma') N_{\gamma \gamma'} 
\qquad \gamma, \gamma' \in \Gamma_x .
\end{equation}
It is assumed that $\natural_{\gamma x}$ and $\gamma_*\natural_x$ define the same class 
in $H^2 (\Gamma_{\gamma x} , \C^\times)$, where $\gamma_* : \Gamma_x \to \Gamma_{\gamma x}$
sends $\alpha$ to $\gamma \alpha \gamma^{-1}$. We define 
\[
\widetilde X_\natural : = \{(x,\rho) : x \in X, \rho \in \Irr \,\C[\Gamma_x, \natural_x ] \}.
\]
and we topologize it by decreeing that a subset of $\widetilde X_\natural$ is open if
and only if its projection to the first coordinate is open in $X$.

We  require, for every $(\gamma,x) \in \Gamma \times X$, a definite algebra isomorphism
\[
\phi_{\gamma,x} : \C[\Gamma_x,\natural_x ]  \to \C[\Gamma_{\gamma x},\natural_{\gamma x}]
\]
such that:
\begin{itemize}
\item if $\gamma x = x$, then $\phi_{\gamma,x}$ is conjugation by an element of
$\C [\Gamma_x,\natural_x]^\times$;
\item $\phi_{\gamma',\gamma x} \circ \phi_{\gamma,x} = 
\phi_{\gamma' \gamma,x}$ for all $\gamma',\gamma \in \Gamma, x \in X$.
\end{itemize}
Then we can define a $\Gamma$-action on $\widetilde X_\natural$ by
\[
\gamma \cdot (x,\rho) = (\gamma x, \rho \circ \phi_{\gamma,x}^{-1}).
\]
We form the \emph{twisted extended quotient}
\[
(X\q \Gamma)_\natural : = \widetilde{X}_\natural/\Gamma.
\]
Notice that the data used to construct this are very similar to a 2-cocycle $z$ of 
$\Gamma$ with values in the continuous functions $X \to \C^\times$. By formulating it in 
the above way, we remove the need to define $z(\gamma,\gamma')$ at points of $X$ that
are not fixed by $\gamma$.

Furthermore we note that $(X \q \Gamma)_\natural$ reduces to the extended quotient 
of the second kind $(X \q \Gamma)_2$ from \cite[\S 2]{ABPS5} if $\natural_x$ is 
trivial for all $x \in X$ and $\phi_{\gamma,x}$ is conjugation by $\gamma$.

The extended quotient of the second kind is an extension of the ordinary quotient
in the sense that it keeps track of the duals of the isotropy groups. Namely, in 
$(X \q \Gamma)_2$ every point $x \in X / \Gamma$ has been replaced by the set 
$\Irr (\Gamma_x)$.

In the context of representation theory, the twisted extended quotient comes into play 
when reducibility at a point is \emph{less than expected}. To be precise, the number 
of inequivalent irreducible representations at a point is fewer than expected.
\vspace{2mm}

\textbf{Example.}
Let $\Gamma = \{ \pm 1 \}^2$, acting on the square $X = [-1,1]^2$ by sign changes of
the coordinates. In the extended quotient $(X \q \Gamma)_2$ we have two points laying
over $(x,0)$ and over $(0,y)$, since $\Gamma_{(x,0)} \cong \Gamma_{(0,y)} \cong \Z / 2 \Z$.
The fiber over $(0,0)$ even has four points, because
\[
\C [\Gamma_{(0,0)}] = \C [\Gamma] \cong \C^4 .
\]
We define a nontrivial 2-cocycle of $\Gamma$ as follows. Define a projective 
$\Gamma$-representation $\lambda$ on $\C^2$ by
\[
\lambda (1,1) = \matje{1}{0}{0}{1},\; \lambda (1,-1) = \matje{i}{0}{0}{-i},\; 
\lambda (-1,1) = \matje{0}{-1}{1}{0} ,\; \lambda (-1,-1) = \matje{0}{i}{-i}{0}.
\]
The cocycle, with values in $\{\pm 1\}$, is given by
\[
\lambda (\gamma) \lambda (\gamma') = 
\natural_x (\gamma,\gamma') \, \lambda (\gamma \gamma') .
\]
In the twisted extended quotient $(X \q \Gamma)_\natural$ the fiber over $(0,0)$ is
in bijection with the set of irreducible representations of
\[
\C [\Gamma_{(0,0)},\natural_{(0,0)}] = \C [\Gamma,\natural] \cong M_2 (\C) , 
\]
so this fiber consists of a single point. The quotients of $X$ by $\Gamma$ look like:\\

\noindent
ordinary quotient \hspace{5mm} (untwisted) extended quotient \hspace{5mm} 
twisted extended quotient \\
\includegraphics[width=12cm]{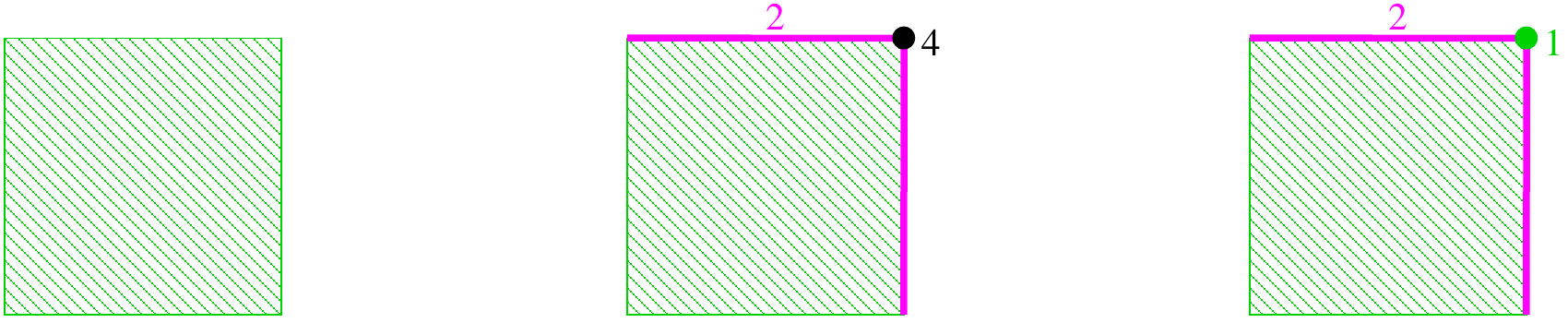} 
\\

More generally, twisted extended quotients arise in the following situation.
Let $A$ be a $\C$-algebra such that all irreducible $A$-modules have countable
dimension over $\C$. Let $\Gamma$ be a group acting on $A$ by
automorphisms and form the crossed product $A \rtimes \Gamma$.

Let $X = \Irr (A)$. Now $\Gamma$ acts on $\Irr (A)$ and we get $\natural$ 
as follows. Given $x \in \Irr (A)$ choose an irreducible representation  
$(\pi_x,V_x)$ whose isomorphism class is $x$. 
For each $\gamma \in \Gamma$ consider $\pi_x$ twisted by $\gamma$:
\[
\gamma \cdot \pi_x : a \mapsto \pi_x (\gamma^{-1} a \gamma).
\]
Then $\gamma \cdot x$ is defined as the isomorphism class of $\gamma \cdot \pi_x$.
Since $\gamma \cdot \pi_x$ is equivalent to $\pi_{\gamma x}$, there exists 
a nonzero intertwining operator 
\begin{equation}\label{eq:B.1}
T_{\gamma,x} \in \Hom_A (\gamma \cdot \pi_x , \pi_{\gamma x}) .
\end{equation}
By Schur's lemma (which is applicable because $\dim V_x$ is countable) $T_{\gamma,x}$ 
is unique up to scalars, but in general there is no preferred choice. 
For $\gamma, \gamma' \in \Gamma_x$ there exists a unique $c \in \C^\times$ such that
\begin{equation}\label{eq:2.21}
c T_{\gamma,x} \circ T_{\gamma',x} = T_{\gamma \gamma',x}.
\end{equation}
We define the 2-cocycle by 
\[
\natural_x (\gamma,\gamma') = c.
\]
Notice the difference between \eqref{eq:2.20} and \eqref{eq:2.21}.
Let $N_{\gamma,x}$ with $\gamma \in \Gamma_x$ be the standard basis of 
$\C [\Gamma_x,\natural_x ]$. The algebra homomorphism $\phi_{\gamma,x}$ is essentially 
conjugation by $T_{\gamma,x}$, but we must be careful if some of the $T_\gamma$
coincide. The precise definition is 
\begin{equation}\label{eq:B.3}
\phi_{\gamma,x} (N_{\gamma',x}) = \lambda^{-1} N_{\gamma \gamma' \gamma^{-1},\gamma x} 
\quad \text{if} \quad T_{\gamma,x} T_{\gamma',x} T_{\gamma,x}^{-1} = 
\lambda T_{\gamma \gamma' \gamma^{-1}, \gamma x}, \lambda \in \C^\times .
\end{equation}
Suppose that $\Gamma_x$ is finite and $(\tau,V_\tau) \in \Irr (\C [\Gamma_x,\natural_x ])$.
Then $V_x \otimes V_\tau$ is an irreducible $A \rtimes \Gamma_x$-module,
where $\gamma \in \Gamma_x$ acts as $T_{\gamma,x} \otimes \tau (N_{\gamma,x})$.

\begin{lem} \label{lem:B.6} 
\textup{\cite[Lemma 2.3]{ABPS5}} \ \\
Let $A$ and $\Gamma$ be as above and assume that the action of $\Gamma$ on
$\Irr (A)$ has finite isotropy groups.
\enuma{
\item There is a bijection
\[
\begin{array}{ccc}
(\Irr (A) \q \Gamma)_\natural & \longleftrightarrow & \Irr (A \rtimes \Gamma) \\
(\pi_x,\tau) & \mapsto & \pi_x \rtimes \tau := 
\Ind_{A \rtimes \Gamma_x}^{A \rtimes \Gamma} (V_x \otimes V_\tau) .
\end{array}
\]
\item If all irreducible $A$-modules are one-dimensional, then part (a) 
becomes a natural bijection
\[
(\Irr (A) \q \Gamma)_2 \longleftrightarrow \Irr (A \rtimes \Gamma) .
\]
}
\end{lem}

Via the following result twisted extended quotients also arise from algebras of
invariants.

\begin{lem}\label{lem:B.7}
Let $\Gamma$ be a finite group acting on a $\C$-algebra $A$. There is a bijection 
\[
\begin{array}{ccc}
\{ V \in \Irr (A \rtimes \Gamma) : V^\Gamma \neq 0 \} & \longleftrightarrow & 
\Irr (A^\Gamma) \\
V & \mapsto & V^\Gamma .
\end{array}
\]
If all elements of $\Irr (A)$ have countable dimension, it becomes
\[
\begin{array}{ccc}
\{ (\pi_x,\tau) \in (\Irr (A) \q \Gamma)_\natural : (V_x \otimes V_\tau)^{\Gamma_x} \neq 0 \} 
& \longleftrightarrow & \Irr (A^\Gamma) \\
(\pi_x,\tau) & \mapsto & (V_x \otimes V_\tau)^{\Gamma_x} .
\end{array}
\]
\end{lem}
\begin{proof}
Consider the idempotent 
\begin{equation}\label{eq:B.2}
p_\Gamma = |\Gamma|^{-1} \sum\nolimits_{\gamma \in \Gamma} \gamma \in \C [\Gamma] .
\end{equation}
It is well-known and easily shown that 
\[
A^\Gamma \cong p_\Gamma (A \rtimes \Gamma) p_\Gamma
\]
and that the right hand side is Morita equivalent with the two-sided ideal 
\[
I = (A \rtimes \Gamma) p_\Gamma (A \rtimes \Gamma) \subset A \rtimes \Gamma .
\]
The Morita equivalence sends a module $V$ over the latter algebra to
\[
p_\Gamma (A \rtimes \Gamma) \otimes_{(A \rtimes \Gamma) p_\Gamma
(A \rtimes \Gamma)} V = V^\Gamma . 
\]
As $I$ is a two-sided ideal, 
\[
\Irr (I) = \{ V \in \Irr (A \rtimes \Gamma) : I \cdot V \neq 0 \}  =
\{ V \in \Irr (A \rtimes \Gamma) : p_\Gamma V = V^\Gamma \neq 0 \}
\]
This gives the first bijection. From Lemma \ref{lem:B.6}.a we know that
every such $V$ is of the form $\pi_x \rtimes \tau$. With Frobenius reciprocity
we calculate
\[
(\pi_x \rtimes \tau)^\Gamma = \Big( \Ind_{A \rtimes \Gamma_x}^{A \rtimes \Gamma} 
(V_x \otimes V_\tau) \Big)^\Gamma \cong (V_x \otimes V_\tau)^{\Gamma_x}.
\]
Now Lemma \ref{lem:B.6}.a and the first bijection give the second.
\end{proof}

Let $A$ be a commutative $\C$-algebra all whose irreducible representations 
are of countable dimension over $\C$. Then $\Irr (A)$ consists of
characters of $A$ and is a $T_1$-space. Typical examples are $A = C_0 (X)$
(with $X$ locally compact Hausdorff), $A = C^\infty (X)$ (with $X$ a smooth
manifold) and $A = \mc O (X)$ (with $X$ an algebraic variety).
As a kind of converse to Lemmas \ref{lem:B.6} and \ref{lem:B.7}, we show that
many twisted extended quotients of $\Irr (A)$ appear as the space of 
irreducible representations of some algebra.

Let $\Gamma$ be a finite group acting on $A$ by algebra automorphisms. 
Let $\tilde \Gamma$ be a central extension of $\Gamma$ and let $\chi_\natural$
be a character of $Z := \ker (\tilde \Gamma \to \Gamma)$. For any (setwise)
section $\lambda : \Gamma \to \tilde \Gamma$, we get a 2-cocycle 
$\natural : \Gamma \times \Gamma \to \C^\times$ by
\[
\lambda (\gamma) \lambda (\gamma') = 
\natural (\gamma,\gamma') \lambda (\gamma \gamma') .
\]
In fact, up to coboundaries every 2-cocycle of $\Gamma$ arises in this way
\cite[\S 53]{CuRe}. Let 
\[
p_\natural := |Z|^{-1} \sum\nolimits_{z \in Z} \chi_\natural (z)^{-1} z \quad \in \C [Z]
\]
be the idempotent associated to $\chi_\natural$. It is central in 
$\C [\tilde \Gamma]$ and
\begin{equation}\label{eq:2.9}
p_\natural \C [\tilde \Gamma] \cong \C [\Gamma,\natural] .
\end{equation}
Lift the action of $\Gamma$ on $A$ to $\tilde \Gamma$ via the given projection.
The algebra $A \rtimes \tilde \Gamma = A \rtimes \C [\tilde \Gamma]$ contains
$A \rtimes p_\natural \C [\tilde \Gamma]$ as a direct summand. 

\begin{lem}\label{lem:B.1}
There is a bijection
\[
\begin{array}{ccc}
( \Irr (A) \q \Gamma )_\natural & \longleftrightarrow & 
\Irr ( A \rtimes p_\natural \C [\tilde \Gamma] ) \\
(\C_x , \tau) & \mapsto & \Ind_{A \rtimes \tilde{\Gamma}_x}^{A 
\rtimes \tilde \Gamma}(\C_x \otimes V_\tau) .
\end{array} 
\]
\end{lem}
\begin{proof}
Start with Lemma \ref{lem:B.6}.b for $A$ and $\tilde \Gamma$.
Since $Z$ acts trivially on $A$, it is contained in $\tilde{\Gamma}_x$ for
every $x \in \Irr (A)$. Now restrict to representations on which $Z$ acts
by $\chi_\natural$.
\end{proof}

\subsection{The Bernstein decomposition} \ 

We return to our reductive $p$-adic group $G$. Recall that an irreducible (smooth, 
complex) $G$-representation is called supercuspidal if it does not appear in any 
$G$-re\-pre\-sen\-ta\-tion induced from a proper Levi subgroup of $G$. Bernstein
\cite[\S 2]{BeDe} realised that an irreducible $G$-representation is supercuspidal
if and only if it is compact. Here compact means that the representation behaves
like one of a compact group, in the sense that all its matrix coefficients have
compact support modulo the centre of $G$. This observation enabled him to prove
that the supercuspidal representations generate a direct factor of the category
of smooth $G$-representations $\Rep (G)$.

That constitutes the first and most important step towards the Bernstein decomposition,
which we describe next. Let $P$ be a parabolic subgroup of $G$ and let $L$ be
a Levi factor of $P$. Let $\omega$ be a supercuspidal $L$-representation. (By
definition this entails that $\omega$ is irreducible.) We call $(L,\omega)$ a
cuspidal pair, and we consider such pairs up to inertial equivalence. This is
the equivalence relation generated by:
\begin{itemize}
\item unramified twists, $(L,\omega) \sim (L,\omega \otimes \chi)$ for $\chi \in 
X_\nr (L)$, where $X_\nr (L)$ is the group of unramified (not necessarily unitary)
characters $L \to \C^\times$;
\item $G$-conjugation, $(L,\omega) \sim (g L g^{-1},g \cdot \omega)$ for $g \in G$.
\end{itemize}
We denote a typical inertial equivalence class by $\fs = [L,\omega]_G$.
In particular 
\[
\fs_L := [L,\omega]_L = \{ \omega \otimes \chi \in \Irr (L) : \chi \in X_\nr (L) \} . 
\]
From $\fs$ Bernstein built a block in the category of smooth $G$-representations, 
in the following way. Denote the normalized parabolic induction functor by $I_P^G$. 
We define
\begin{align*}
& \Irr (G)^\fs = \{ \pi \in \Irr (G) : \pi \text{ is a constituent of } I_P^G (\omega 
\otimes \chi) \text{ for some } \omega \in \fs_L \} , \\
& \Rep (G)^\fs = \{ \pi \in \Rep (G) : \text{ every irreducible constituent of } \pi
\text{ belongs to } \Irr (G)^\fs \} .
\end{align*}
We denote the set of all inertial equivalence classes for $G$ by $\Omega (G)$.

\begin{thm}\label{thm:2.1} \textup{\cite[Proposition 2.10]{BeDe}} \\
The category of smooth $G$-representations decomposes as
\[
\Rep (G) = \prod\nolimits_{\fs \in \Omega (G)} \Rep (G)^\fs . 
\]
The space of irreducible $G$-representations is a disjoint union
\[
\Irr (G) = \bigsqcup\nolimits_{\fs \in \Omega (G)} \Irr (G)^\fs .
\]
\end{thm}

Let $\Irr_\cusp (L)$ be the set of supercuspidal $L$-representations, up to
isomorphism. For $\omega \in \Irr_\cusp (L)$ (and in fact for every irreducible
$L$-representation) the group
\[
X_\nr (L,\omega) := \{ \chi \in X_\nr (L) : \omega \otimes \chi \cong \omega \}
\]
is finite. Thus there is a bijection
\begin{equation}\label{eq:2.1}
X_\nr (L) / X_\nr (L,\omega) \to \Irr (L)^{\fs_L} : \chi \mapsto \omega \otimes \chi ,
\end{equation}
which endows $\Irr (L)^{\fs_L}$ with the structure of a complex torus. Up to 
isomorphism this torus depends only on $\fs$, and it is known as the Bernstein 
torus $T_\fs$. We note that $T_\fs$ is only an algebraic variety, it is not endowed
with a natural multiplication map. In fact it does not even possess an unambigous
``unit'', because in general there is no preferred choice of an element $\omega \in
\fs_L$. 

Consider $W(G,L) = N_G (L) / L$, the ``Weyl'' group of $(G,L)$. It acts on 
$\Irr (L)$ by
\begin{equation}\label{eq:2.4}
w \cdot \pi = [ \bar w  \cdot \pi : l \mapsto \pi (\bar{w}^{-} l \bar w ) ]
\quad \text{for any lift } \bar w \in N_G (L) \text{ of } w \in W(G,L) .
\end{equation}
To $\fs$ Bernstein also associated the finite group
\begin{equation} \label{eqn: Ws}
W_\fs := \{ w \in W(G,L) : w \cdot \Irr (L)^{\fs_L} = \Irr (L)^{\fs_L} \} .
\end{equation}
It acts naturally on $T_\fs$, by automorphisms of algebraic varieties.

Closely related to the Bernstein decomposition is the theory of the Bernstein
centre. By \cite[Th\'eor\`eme 2.13]{BeDe} the categorical centre of the 
Bernstein block $\Rep^\fs (G)$ is
\begin{equation}\label{eq:2.2}
Z( \Rep (G)^\fs) \cong \mc O (T_\fs )^{W_\fs} = \mc O (T_\fs / W_\fs) . 
\end{equation}
Here $\mc O$ stands for the regular functions on an affine variety.
Moreover the map 
\begin{equation}\label{eq:2.3}
\text{\bf sc} : \Irr (G)^\fs \to T_\fs / W_\fs
\end{equation}
induced by \eqref{eq:2.2} is surjective and has finite fibers \cite[\S 3]{BeDe}.
Theorem \ref{thm:2.1} implies that every $\pi \in \Irr (G)$ is a constituent
of $I_P^G (\omega)$, where $[L,\omega]_G$ is uniquely determined. By 
\eqref{eq:2.2} the supercuspidal $L$-representation $\omega \in T_\fs$ is in fact 
uniquely determined up to $W_\fs$. The map $\pi \mapsto W_\fs \omega$ is just 
\textbf{sc}, and for this reason it is called the cuspidal support map.
Via this map $\Irr^\fs (G)$ can be regarded as a non-separated algebraic
variety lying over $T_\fs / W_\fs$.

\subsection{Geometric structure of Bernstein components} \

Let $\fs = [L,\omega]_G$ be an inertial equivalence class for $G$. Based on
many examples, we believe that the geometric structure of the component
$\Irr^\fs (G)$ of $\Irr (G)$ is related to its Bernstein centre $\mc O 
(T_\fs / W_\fs)$ in a strikingly simple and precise way.

Let $W_{\fs,t}$ be the stabilizer in $W_\fs$ of a point $t \in T_\fs$.

\begin{conj}\label{conj:2}
There exists a family of 2-cocycles
\[
\natural_t : W_{\fs,t} \times W_{\fs,t} \to \C^\times \qquad t \in T_\fs ,
\]
and a bijection
\[
\Irr (G)^\fs \longleftrightarrow (T_\fs \q W_\fs )_\natural 
\]
such that:
\begin{itemize}
\item It restricts to a bijection between tempered representations and 
the unitary part of the extended quotient (as explained below).
\item The bijection is canonical up to permutations within L-packets. That
is, for any $\phi \in \Phi (G)$, the image of $\Pi_\phi (G) \cap \Irr^\fs (G)$ 
is canonically defined (assuming a LLC for $G$ exists).
\end{itemize}
\end{conj}

Let $\Irr_\cusp (L)$ be the set of supercuspidal $L$-representations.
It is stable under the $W(G,L)$-action \eqref{eq:2.4}. The definitions of
$W_\fs$ and of extended quotients imply that for a fixed Levi subgroup $L$
of $G$ there is a natural bijection
\begin{equation}
\bigsqcup\nolimits_{\fs = [L,\omega]_G} (T_\fs \q W_\fs )_\natural \to
\Big( \Irr_\cusp (L) \q W(G,L) \Big)_\natural .
\end{equation}
In view of Theorem \ref{thm:2.1}, Conjecture \ref{conj:2} can also be
formulated, more elegantly, in terms of a bijection
\begin{equation}\label{eq:2.8}
\Irr (G) \longleftrightarrow \bigsqcup\nolimits_L 
\Big( \Irr_\cusp (L) \q W(G,L) \Big)_\natural ,
\end{equation}
where $L$ runs through a set of representatives for the $G$-conjugacy classes
of Levi subgroups of $G$. In this version, our conjecture asserts that
$\Irr (G)$ is determined by a much smaller set of data, namely the supercuspidal 
representations of Levi subgroups $L$ of $G$, and the actions of the Weyl 
groups $W(G,L)$ on those.

We expect that the group cohomology classes $\natural_t \in H^2 (W_{\fs,t},
\C^\times)$ reflect the character of $Z (G^\vee_{\sc})^{\mb W_F}$ which via
the Kottwitz isomorphism \eqref{eq:1.8} determines how $G$ is an inner twist
of a quasi-split group. In particular $\natural$ should be trivial whenever $G$
is quasi-split. The simplest known example of a nontrivial cocycle involves a
non-split inner form of $\SL_{10}(F)$ \cite[Example 5.5]{ABPS4}. That example
also shows that it is sometimes necessary to use \emph{twisted} extended 
quotients in Conjecture \ref{conj:2}.

Recall \cite[\S III.1--III.2]{Wal} that a supercuspidal representation is 
tempered if and only if it is unitary. Let $T_{\fs,un}$ be the set of unitary
representations in $T_\fs$, a $W_\fs$-stable compact real subtorus. Let us
denote the group of unitary unramified characters of $L$ by $X_\unr (L)$.
Without loss of generality we may assume that the basepoint $\omega \in T_\fs$ 
is unitary. Then \eqref{eq:2.1} becomes a bijection
\[
X_\unr (L) / X_\nr (L,\omega) \to T_{\fs,un} : \chi \mapsto \omega \otimes \chi . 
\]
Let $X_\nr^+ (L)$ be the group of unramified characters $L \to \R_{>0}$. 
The polar decomposition of $X_\nr (L)$ reads
\[
X_\nr (L) = X_\unr (L) \times X_\nr^+ (L) . 
\]
Since $X_\nr (L,\omega)$ is finite and $\R_{>0}$ has no nontrivial finite 
subgroups, $X_\nr (L,\omega) \cap X_\nr^+ (L) = \{1\}$. Hence the canonical map
\begin{equation}\label{eq:2.5}
T_{\fs,un} \times X_\nr^+ (L) \to T_\fs : 
(\sigma,\chi^+) \mapsto \sigma \otimes \chi^+
\end{equation}
is bijective. We regard \eqref{eq:2.5} as the polar decomposition of $T_\fs$.

Let $\Irr_\temp (G)$ be the set of irreducible tempered $G$-representations
(still considered up to isomorphism) and write
\[
\Irr_\temp (G)^\fs = \Irr (G)^\fs \cap \Irr_\temp (G) . 
\]
Conjecture \ref{conj:2} asserts that there is a bijection
\begin{equation}\label{eq:2.7}
\Irr_\temp (G)^\fs \longleftrightarrow (T_{\fs,un} \q W_\fs )_\natural . 
\end{equation}
In view of the $W_\fs$-equivariant polar decomposition \eqref{eq:2.5},
$(T_\fs \q W_\fs )_\natural$ is a natural way the complexification of its
compact real form $(T_{\fs,un} \q W_\fs )_\natural$. Similarly $\Irr^\fs (G)$ can 
be regarded as the ``complexification'' of $\Irr_\temp (G)^\fs$ \cite[\S 2]{ABPS1}.
If we manage to construct a bijection \eqref{eq:2.7} with suitable properties, then 
the method of \cite[\S 4]{ABPS1} shows that it extends to a bijection
$\Irr (G)^\fs \longleftrightarrow (T_\fs \q W_\fs )_\natural$ with the same
properties. Thus it suffices to prove Conjecture \ref{conj:2} for tempered 
representations.\\

\textbf{Example.}
Consider $G = \GL_2 (F)$ with the standard diagonal torus $T$. Let
$\fs = [T,\mathrm{triv}_T]_G$. Then 
\[
T_\fs = X_\nr (T) \cong (\C^\times )^2
\]
and $W_\fs = \{ 1, \matje{0}{1}{1}{0} \}$, acting on $T_\fs$ by permutations
of the two coordinates. In this case all the 2-cocycles $\natural_t$ are trivial
and the extended quotient is
\[
( T_\fs \q W_\fs )_2 = T_\fs / W_\fs \times \{\mathrm{triv} \} 
\; \sqcup \; \{ ((z,z),\mathrm{sign}_{W_\fs}) : z \in \C^\times \}
\]
The bijection from Conjecture \ref{conj:2} is canonical:
\[
\begin{array}{c@{\; \longleftrightarrow \;}cl}
\Irr (G)^\fs & ( T_\fs \q W_\fs )_2  \\
I_B^G (z,z') & ((z,z'), \mathrm{triv}) & 
\hspace{1cm} z' \in \C^\times \setminus \{q_F z,q_F^{-1} z\} \\
L (I_B^G (q_F z,z)) & ((q_F z,z) , \mathrm{triv}) \\
St_G \otimes z^{\nu_F \circ \det} & ((z,z),\mathrm{sign}_{W_\fs})
\end{array}
\]
The description of $\Irr (\GL_2 (F))^\fs$ is well-known, a clear account
of it can be found in \cite[\S 17]{BuHe}. To write it down we used
\[
\begin{array}{lll}
B & = & \text{standard Borel subgroup, the upper triangular matrices in } \GL_2 (F), \\
q_F & = & |k_F| \text{, cardinality of the residue field of } F, \\
L(\pi) & = & \text{Langlands quotient of the parabolically induced representation } \pi, \\
St_G & = & \text{Steinberg representation of } G , \\
\nu_F & = & \text{discrete valuation of the field } F .
\end{array}
\]
\textbf{Example.}
Take $G = \SL_2 (F)$, and the other notations as above but for $\SL_2 (F)$. Now
\[
T_\fs \to \C^\times : \chi \mapsto \chi \Big( \matje{\varpi_F}{0}{0}{\varpi_F^{-1}} \Big) 
\]
is a bijection, for any uniformizer $\varpi_F$ of $F$. The group $W_\fs = \{1,w\}$
acts on $T_\fs$ by $w \cdot z = z^{-1}$. The relevant extended quotient is
\[
(T_\fs \q W_\fs )_2  = T_\fs / W_\fs \times \{\mathrm{triv} \} 
\; \sqcup \; \{ (\pm 1,\mathrm{sign}_{W_\fs})  \}
\]
It is in bijection with $\Irr (G)^\fs$ via
\[
\begin{array}{c@{\; \longleftrightarrow \;}cl}
\Irr (G)^\fs & ( T_\fs \q W_\fs )_2 \\
I_B^G (z) & (z, \mathrm{triv}) & 
\hspace{1cm} z \in \C^\times \setminus \{-1,q_F ,q_F^{-1}\} \\
L (I_B^G (q_F)) & (q_F , \mathrm{triv}) \\
St_G & (1,\mathrm{sign}_{W_\fs}) \\
I_B^G (-1) = \pi_+ \oplus \pi_- & 
\{ (-1,\mathrm{triv}_{W_\fs}),(-1,\mathrm{sign}_{W_\fs}) \}
\end{array}
\]
Notice that the unramified character $\matje{a}{0}{0}{a^{-1}} \mapsto (-1)^{\nu_F (a)}$
gives rise to an L-packet with two irreducible $G$-representations, denoted $\pi_\pm$.
Both must be mapped to a point in the extended quotient, lying over 
$-1 \in T_\fs / W_\fs$. There are two ways to do so, both equally good. There does not 
seem to be a canonical choice without specifying additional data, see 
\cite[Example 11.3]{ABPS5}.
\vspace{3mm}

At the time of writing, Conjecture \ref{conj:2} has been proven in the following cases.
\begin{itemize}\label{casesABPS}
\item General linear groups over division algebras \cite{ABPS4,ABPS6}.
\item Special linear groups over division algebras \cite{ABPS4,ABPS6}.
\item Split orthogonal and symplectic groups \cite[\S 5]{Mou}. 
\item Principal series representations of split groups \cite{ABPS7}, 
\cite[\S 18--19]{ABPS5}.
\end{itemize}

\subsection{Hecke algebras for Bernstein blocks} \

We will explain some of the ideas that lead to the proof of Conjecture \ref{conj:2}
in the aforementioned cases. Let $\cH (G)$ be the Hecke algebra of $G$, that is,
the vector space $C_c^\infty (G)$ of locally constant compactly supported functions
on $G$, endowed with the convolution product. It is the version of the group algebra
of $G$ which is most suitable for studying smooth representations. The category
$\Rep (G)$ is naturally equivalent with the category $\Rep (\cH (G))$ of 
$\cH (G)$-modules $V$ such that $\cH (G) \cdot V = V$. (The latter condition is
nontrivial because $\cH (G)$ does not have a unit if $G \neq 1$.)

In these terms the Bernstein decomposition becomes
\begin{equation}\label{eq:2.10}
\begin{array}{lll}
\cH (G) & = & \bigoplus_{\fs \in \Omega (G)} \cH (G)^\fs , \\
\Rep (G) & \cong & \bigoplus_{\fs \in \Omega (G)} \Rep (\cH (G)^\fs ) , \\
\Irr (G) & = & \bigsqcup_{\fs \in \Omega (G)} \Irr (\cH (G)^\fs ) .
\end{array} 
\end{equation}
In other words, $\Rep (\cH (G)^\fs )$ is a Bernstein block for $G$.
Unfortunately, the algebras $\cH (G)^\fs$ are in general too large to work well with.
To perform interesting computations, one has to downsize them. The most common
approach is due to Bushnell and Kutzko \cite{BuKu1,BuKu2}. They propose to look
for suitable idempotents $e_\fs \in \cH (G)$ such that:
\begin{itemize}
\item $\cH (G)^\fs = \cH (G) e_\fs \cH (G)$, and this is Morita equivalent with 
$e_\fs \cH (G) e_\fs$ via the map $V \mapsto e_\fs V$;
\item $e_\fs \cH (G) e_\fs$ is smaller and simpler than $\cH (G)^\fs$. 
\end{itemize}
Typically $e_\fs$ will be associated to an irreducible representation of a compact
open subgroup of $G$, then Bushnell and Kutzko call it a type for $\fs$. Yet
in some cases this might be asking for too much, so we rather not require that.

The challenge is to find an idempotent such that the structure of $e_\fs \cH (G) e_\fs$
is nice and explicit. Let us call such an $e_\fs$ a nice idempotent for $\fs$. 
In practice this means that $e_\fs \cH (G) e_\fs$ must be close to an affine Hecke 
algebra. Such algebras can be defined in several ways \cite{IwMa,Lus2}, here we
present a construction which is well-adapted to representations of $p$-adic groups.
Let $T$ a complex torus with character lattice $X^* (T)$. Let $R \subset X^* (T)$
be a root system, not necessarily reduced. The Weyl group $W(R)$ acts on $T, X^* (T), 
\mc O (T)$ and $R$. We also need a parameter function $q: R / W(R) \to \R_{>0}$.

\begin{defn}
The affine Hecke algebra $\cH (T,R,q)$ is the $\C$-algebra such that:
\begin{itemize}
\item As vector space it equals $\mc O (T) \otimes \C [W(R)]$.
\item $\mc O (T)$ is embedded as a subalgebra.
\item $\C [W(R)] = \mathrm{span}\{ N_w : w \in W(R) \}$ is embedded as the
Iwahori--Hecke algebra $\cH (W(R),q)$, that is, the multiplication is defined by
\[
\begin{array}{ll}
N_w N_v = N_{wv} & \text{if } \ell (w) + \ell (v) = \ell (wv) ,\\
\big( N_{s_\alpha} - q_\alpha^{1/2} \big) \big( N_{s_\alpha} + q_\alpha^{-1/2} \big) = 0 &
\text{for every simple reflection } s_\alpha .
\end{array}
\]
Here $\ell$ is the length function of $W(R)$ and $\alpha \in R$ is a simple root.
\item The commutation rules between $\mc O (T)$ and $\cH (W (R),q)$ are determined by
\[
f N_{s_\alpha} - N_{s_\alpha} s_\alpha (f) = 
(q_\alpha^{1/2} - q_\alpha^{-1/2}) \frac{f - s_\alpha (f)}{1 - \theta_{-\alpha}} .
\]
Here $f \in \mc O (T), \alpha$ is a simple root and $\theta_x \in \mc O (T)$ 
corresponds to $x \in X^* (T)$. (In fact the formula can be slightly more complicated
if $R$ contains a factor of type $C_l$, see \cite[\S 3]{Lus2}.)
\end{itemize}
\end{defn}

Notice that for the parameter function $q = 1$ we get 
\begin{equation}
\cH (T,R,1) = \mc O (T) \rtimes W(R) = \C [X^* (T) \rtimes W(R)] . 
\end{equation}
With Lemma \ref{lem:B.6}.b we obtain a natural bijection
\[
\Irr (\mc O (T) \rtimes W(R)) \longleftrightarrow (T \q W(R) )_2 . 
\]
The representations of affine Hecke algebras have been subjected to a lot of study,
see in particular \cite{Lus2,KaLu,Opd,Sol2}. As a result the representation theory
of $\cH (T,R,q)$ is understood quite well, and close relations between
$\Irr (\cH (T,R,q))$ and $\Irr (\cH (T,R,1)) \cong (T \q W(R) )_2$ are known.
This is the main source of extended quotients in the representation theory of
reductive $p$-adic groups.

Now we provide an overview of what is known about the structure of 
$e_\fs \cH (G) e_\fs$ in various cases.\\

\textbf{Iwahori--spherical representations.} \\
This is the classical case. Let $M$ be a minimal Levi subgroup of $G$  and 
$\fs = [M,\mathrm{triv}_M]_G$. Borel \cite{Bor1} showed that the idempotent
$e_I$ associated to an Iwahori subgroup $I$ is nice for $\fs$. By
\cite[\S 3]{IwMa} there is an algebra isomorphism 
\begin{equation}\label{eq:2.11}
C_c (I \backslash G / I) \cong e_I \cH (G) e_I \cong \cH (X_\nr (M),
R^\vee (G,M), q_I) ,
\end{equation}
where $R^\vee (G,M)$ is the system of coroots of $G$ with respect to the maximal
split torus in $Z(M)$ and $q_{I,\alpha} = \mathrm{vol}(I s_\alpha I) / 
\mathrm{vol}(I)$ for a simple reflection $s_\alpha$.\\

\textbf{Principal series representations of split groups.} \\
Suppose that $G$ is $F$-split and let $T$ be a maximal split torus of $G$.
Fix a smooth character $\chi_\fs \in \Irr (T)$ and put $\fs = [T,\chi_\fs]_G$,
so that 
\[
X_\nr (T) \to T_\fs : \chi \mapsto \chi \chi_\fs 
\]
is a homeomorphism. By \cite[Lemma 6.2]{Roc} there exist a root subsystem
$R_\fs \subset R^\vee (G,T)$ and a subgroup $\mf R_\fs \subset W_\fs$ such that
$W_\fs = W (R_\fs) \rtimes \mf R_\fs$.

\begin{thm}\label{thm:Roche} \textup{\cite[Theorem 6.3]{Roc}} \\
There exists a type for $\fs$ and an algebra isomorphism
\[
e_\fs \cH (G) e_\fs \cong \cH (T_\fs,R_\fs,q) \rtimes \mf R_\fs , 
\]
where $q_\alpha = |k_F|$ for all $\alpha \in R_\fs$.
\end{thm}

\textbf{Level zero representations.} \\
These are $G$-representations which contain non-zero vectors fixed by the
pro-unipotent radical of a parahoric subgroup of $G$. For such representations
the algebra $e_\fs \cH (G) e_\fs$ can be determined via suitable reductive groups
over the residue field $k_F$ \cite[Theorem 7.12]{Mor}, see also \cite{Lus3}.
It turns out that, like Theorem \ref{thm:Roche}, $e_\fs \cH (G) e_\fs$ is of the 
form $\cH (T_\fs,R_\fs,q_\fs) \rtimes \C [\mf R_\fs,\natural_\fs]$ for suitable
$R_\fs, q_\fs$ and $\mf R_\fs$. In all examples of level zero Bernstein blocks
which have been worked out, the 2-cocycle $\natural_\fs$ of $\mf R_\fs$ is trivial. \\

\textbf{Symplectic and orthogonal groups.} \\
For any inertial equivalence class $\fs \in \Omega (G)$ Heiermann \cite{Hei1}
proved that $\cH (G)^\fs$ is Morita equivalent with $\cH (T_\fs,R_\fs,q_\fs)
\rtimes \mf R_\fs$, for suitable $R_\fs, q_\fs$ and $\mf R_\fs$. A type for
$\fs$ was constructed in \cite{MiSt}. It seems plausible that 
$e_\fs \cH (G) e_\fs \cong \cH (T_\fs,R_\fs,q_\fs) \rtimes \mf R_\fs$, but 
as far as we know this has not yet been checked.\\

\textbf{Inner forms of $\GL_n (F)$.} \\
Let $D$ be a division algebra with centre $F$. Every Levi subgroup of $G = \GL_m (D)$
is of the form $L = \prod_i \GL_{m_i}(D)^{e_i}$, where $\sum_i m_i e_i = m$. 
Fix $\omega \in \Irr_\cusp (L)$, of the form $\omega = \bigotimes_{i=1}^k 
\omega_i^{\otimes e_i}$, where $\omega_i \in \Irr_\cusp (\GL_{m_i}(D))$ is not
inertially equivalent with $\omega_j$ if $i \neq j$. Then $T_\fs \cong \prod_{i=1}^k 
(\C^\times)^{e_i}$, $R_\fs$ is of type $\prod_{i=1}^k A_{e_i - 1}$ and
\[
W_\fs = W(R_\fs) \cong \prod\nolimits_{i=1}^k S_{e_i},
\] 
where $W_\fs$ is the group defined in (\ref{eqn: Ws}).
\begin{thm}\label{thm:2.2} \textup{\cite{Sec,SeSt}} \\
There exist a type for $\fs$, a finite dimensional vector space $V$ and a
parameter function $q_\fs : R_\fs \to q^\N$ such that
\[
e_\fs \cH (G) e_\fs \cong \cH (T_\fs,R_\fs,q_\fs) \otimes \End_\C (V) .
\]
\end{thm}

\textbf{Inner forms of $\SL_n (F)$.} \\
Let $G = \SL_m (D)$, the kernel of the reduced norm map $\GL_m (D) \to F^\times$.
Every Levi subgroup of $G$ looks like $L = L' \cap \SL_m (D)$, where 
$L'=\prod_i \GL_{m_i}(D)^{e_i}$. Fix $\fs = [L,\omega]_G$ and choose an 
$\omega' \in \Irr_\cusp (L')$ which contains $\omega$. Then $R_\fs$ is, as above for 
$\omega'$, of type $\prod_{i=1}^k A_{e_i - 1}$, but $T_\fs$ and $W_\fs$ are modified 
compared to $\GL_m (D)$. An explicit description of $T_\fs$ may be found in 
\cite[Prop.~2.1]{ABPS6}. 

Write $M'=\prod_i\GL_{e_i m_i}(D)$ and let $P'$ be the parabolic subgroup of $\GL_m(D)$ 
generated by $L'$ and the upper triangular-block matrices. Then
\[
W_\fs = W(R_\fs) \rtimes \cR_\fs \cong \prod\nolimits_{i=1}^k S_{e_i} \rtimes \cR_\fs,
\]
with $\cR_\fs = W_\fs \cap N_{\GL_m(D)}(P'\cap M')/L'$.

\begin{thm}\label{thm:2.3} \textup{\cite[\S 4.4]{ABPS4}} \\
There exist a finite dimensional projective representation $V$ of $X_\nr (L,\omega)
\rtimes \cR_\fs$ and a nice idempotent $e_\fs$ for $\fs$, such that
\[
e_\fs \cH (G) e_\fs \cong \Big( \cH ( X_\nr (L) ,R_\fs, q_\fs) \otimes \End_\C (V)
\Big)^{X_\nr (L,\omega)} \rtimes \cR_\fs .
\]
Here $X_\nr (L,\omega) \rtimes \cR_\fs$ acts both on $\cH ( X_\nr (L) ,R_\fs, q_\fs)$ 
and on $\End_\C (V)$.
\end{thm}
The algebras appearing in Theorem \ref{thm:2.3} are quite a bit more general than
the previous ones. See \cite[\S 5]{ABPS4} for some examples of what can happen.

For instance, they need not be Morita equivalent to an affine Hecke algebra extended
by a finite group of automorphisms of the root system. That can already happen in
the split case $G = \SL_n (F)$ \cite[\S 11.8]{GoRo}.
Moreover, the projective action of $\cR_\fs$ on $V$ gives rise to a possibly 
nontrivial 2-cocycle of $\cR_\fs$. It is related to the character of
$Z (\SL_n (\C)) = Z( G^\vee_\sc )^{\mb W_F}$ that specifies $G$ as an inner twist
of $\SL_n (F)$, see \cite[Theorem 4.15]{ABPS4}.\\[2mm]

From a more general point of view, the algebra in Theorem \ref{thm:2.3} rather closely
resembles the shape of the Fourier transform of a component in the Schwartz algebra
of any reductive $p$-adic group $G$ \cite{Wal}. The main difference is that for the 
Schwartz algebra one has to replace $\mc O(T_\fs)$ by $C^\infty (T_{\fs,un})$. 

From $\fs = [L,\omega]_G, T_\fs$ and $W_\fs$ one can canonically deduce a root system 
$R_\fs$, namely the set of roots of $(G,Z(L)^\circ)$ for which the Harish--Chandra
$\mu$-function has a pole on $T_\fs$ \cite{H-C}. The group $W_\fs$ acts on the Weyl
chambers for $R_\fs$, and the stabilizer of a fixed positive chamber is a subgroup
$\cR_\fs \subset W_\fs$. Since $W(R_\fs)$ acts simply transitively on the collection 
of Weyl chambers, $W_\fs = W(R_\fs) \rtimes \cR_\fs$.
On the basis of the above, we expect:

\begin{conj}\label{conj:3}
Let $\fs = [L,\omega]_G$ be any inertial equivalence class and use the above notations.
There exist a parameter function $q_\fs : R_\fs \to \R_{>0}$,
a finite dimensional projective representation $V_\fs$ of 
$X_\nr (L,\omega) \rtimes \cR_\fs$, and a nice idempotent $e_\fs$ for $\fs$ such that
\[
e_\fs \cH (G) e_\fs \cong \Big( \cH ( X_\nr (L) ,R_\fs, q_\fs) \otimes \End_\C (V_\fs)
\Big)^{X_\nr (L,\omega)} \rtimes \cR_\fs .
\]
\end{conj}

\subsection{Conjectural construction of the bijection} \

Let us return to Conjecture \ref{conj:2}. Whenever Conjecture \ref{conj:3} holds for 
$\fs$, one can apply \cite[\S 5.4]{Sol2}. This proves an earlier version of Conjecture
\ref{conj:2} for $\Irr^\fs (G)$ (formulated in terms of an extended quotient of the 
first kind, see \cite{ABPS2}). 
To obtain Conjecture \ref{conj:2} completely more work is
required, which has been carried out in the cases listed on page \pageref{casesABPS}.

Based on knowledge of the representation theory of affine Hecke algebras and assuming 
Conjecture \ref{conj:3}, we sketch how the bijection $\Irr (G)^\fs \to 
(T_\fs \q W_\fs )_\natural$ should be constructed. That is, we describe how the 
construction goes in the aforementioned known
cases, and we expect that something similar works in general.

As discussed around \eqref{eq:2.7} it suffices to construct
\begin{equation}\label{eq:2.6}
\Irr_\temp (G)^\fs \to (T_{\fs,un} \q W_\fs)_\natural .
\end{equation}
Let $(\pi ,V_\pi) \in \Irr_\temp (G)^\fs$.
\begin{itemize}
\item As we saw in \eqref{eq:2.3}, the cuspidal support of $\pi$ is an element
$\mathbf{sc}(\pi) \in T_\fs / W_\fs$. Choose a lift $\mathbf{sc}(\pi) \in T_\fs$ and let 
$t = \mathbf{sc}(\pi)_{un} \in T_{\fs,un}$ be its unitary part, obtained from the polar
decomposition \eqref{eq:2.5}. This $t$ will be the $T_\fs$-coordinate in the extended
quotient.
\item Let $e_\fs$ be as in Conjecture \ref{conj:3}, so $e_\fs V_\pi \in 
\Rep (e_\fs \cH (G) e_\fs)$. Recall from \eqref{eq:2.2} that 
\[
Z(e_\fs \cH (G) e_\fs) \cong Z (\Rep (G)^\fs ) \cong \mc O (T_\fs / W_\fs) .
\]
The algebra $e_\fs \cH (G) e_\fs$ contains $\mc O (T_\fs)$ as a subalgebra such that
$Z(e_\fs \cH (G) e_\fs) = \mc O(T_\fs)^{W_\fs}$. All the weights for the action of
$\mc O (T_\fs)$ on $e_\fs V_\pi$ are contained in $W_\fs \mb{sc}(\pi)$, which is a
subset of $W_\fs t X_\nr^+ (L)$. As vector spaces 
\[
e_\fs V_\pi = \bigoplus\nolimits_{w \in W_\fs / W_{\fs,t}} (e_\fs V_\pi )_{w t} , 
\]
where $(e_\fs V_\pi )_{w t}$ is the linear subspace of $e_\fs V_\pi$ on which 
$\mc O (T_\fs)$ acts by weights from $w t X_\nr^+ (L)$.
\item With involved techniques from affine Hecke algebras \cite{Lus2,Sol2} one can
endow $(e_\fs V_\pi )_t$ with a linear action of $W(R_\fs)_t$, the stabilizer of $t$
in $W(R_\fs)$. It extends to a
representation of $\C [W_{\fs,t},\natural_\fs]$, where the 2-cocycle $\natural_\fs$
is determined by the projective $\cR_\fs$-representation $V_\fs$ from Conjecture 
\ref{conj:3}. Define $\natural$ such that $\natural_t = \natural_\fs \big|_{W_{\fs,t}}$.
\item It remains to specify an irreducible representation of $\C [W_{\fs,t},\natural_\fs]$,
depending on $(e_\fs V_\pi )_t$. There are a root subsystem $R_{\fs,t}$ and a Weyl 
subgroup $W(R_{\fs,t}) \subset W(R_\fs)$. The Springer correspondence associates to every 
irreducible $W(R_{\fs,t})$-representation a unipotent orbit in some complex reductive 
group. The dimension of this orbit can be regarded as an invariant, which we call the 
$a$-weight of the representation, where $a$ is the function defined
by Lusztig in \cite{Lus4}. Let $m$ be the maximal $a$-weight 
appearing among the $W(R_{\fs,t})$-subrepresentations of $(e_\fs V_\pi)_t$, and let 
$V_\rho$ be the sum of the $W(R_{\fs,t})$-subrepresentations of $a$-weight $m$. It turns 
out that $(\rho,V_\rho)$ is an irreducible $\C [W_{\fs,t},\natural_\fs]$-representation.
\end{itemize}
Then \eqref{eq:2.6} sends $\pi \in \Irr_\temp (G)^\fs$ to $(t,\rho) \in 
(T_{\fs,un} \q W_\fs )_\natural$.\\

Obviously the construction of $\rho$ is very complicated, and it is hard to see just 
from the above sketch what is going on. We want to make the point that
Conjecture \ref{conj:2} is not about some mysterious bijection, but about a map which
we already know quite well.

Our construction also reveals some (conjectural) information about L-packets. Let
$G^\vee_{\fs,t}$ be (possibly disconnected) complex reductive group with maximal torus
$T_\fs$, root system $R_{\fs,t}$ and Weyl group $W_{\fs,t}$. The extension to $W_{s,t}$ 
of the Springer correspondence for $W(R_{s,t})$,
as in \cite[Theorem 4.4]{ABPS5}, associates to $(\rho,V_\rho)$ 
a unique unipotent class $u(\rho)$ in $G^\vee_{\fs,t}$. It still depends canonically on 
$\pi$, because the $W(R_{\fs,t})$-representation $(e_\fs V_\pi )_t$ does. Only the 
extension of $(e_\fs V_\pi )_t$ to a $\C[W_{\fs,t},\natural_\fs]$-representation need 
not be canonical.

In all examples the L-parameter of $\pi$ depends only on $(t,u(\rho))$, and
$\pi' \in \Irr_\temp (G)^\fs$ has the same L-parameter if and only if 
$W_\fs (t,u(\rho)) = W_\fs (t',u(\rho'))$. Therefore we believe that the bijection in
Conjecture \ref{conj:2} is canonical up to permutations within L-packets.

\section{Reduction to the supercuspidal case}

We discuss a strategy to reduce the construction of a LLC for irreducible smooth
representations to the case of supercuspidal representations. In view of the work
of V. Lafforgue \cite{Laf2,Laf3}, this could be useful in large generality.
(While this paper was under review, the material in this section has been worked
out in \cite{AMS}.)

If one assumes the bijective LLC (Conjecture \ref{conj:1}) for $G$ (considered
as in inner twist of a quasi-split group), then the Bernstein decomposition
of $\Irr (G)$ can be transferred to enhanced L-parameters:
\[
\Phi_e (G) = \bigsqcup\nolimits_{\fs \in \Omega (G)} \Phi_e (G)^\fs , 
\]
where $\Phi_e (G)^\fs$ is the set that parametrizes $\Irr (G)^\fs$. Fixing 
a character $\zeta_G$ of $Z(G^\vee_\sc)$ as in Paragraph \ref{par:bijective},
we obtain a similar decomposition of $\Phi_{e,\zeta_G}(G)$.

If we also assume Conjecture \ref{conj:2} for $\fs = [L,\omega]_G$, then 
$\Irr (G)^\fs$ is in bijection with a twisted extended quotient 
$(T_\fs \q W_\fs )_\natural$. By the conjectural LLC for supercuspidal 
representations of $L$, $T_\fs$ should be in bijection with 
\[
\Phi_{e,\zeta_G} (L)^{\fs_L} := \big\{ (\phi,\rho) \in \Phi_e (L)^{\fs_L} : 
\rho \big|_{Z(L^\vee_\sc)} = \zeta_G \big|_{Z(L^\vee_\sc)} \big\} .
\]
With the fifth desideratum of the LLC for $G$ and $L$, we get bijections
\begin{equation}\label{eq:3.1}
\Phi_{e,\zeta_G} (G)^\fs \longleftrightarrow \Irr (G)^\fs 
\longleftrightarrow (T_\fs \q W_\fs )_\natural
\longleftrightarrow (\Phi_{e,\zeta_G} (L)^{\fs_L} \q W_\fs )_\natural .
\end{equation}
If we can do this for all inertial equivalence classes $\fs \in \Omega (G)$, we 
even obtain a bijection
\[
\Phi_{e,\zeta_G} (G) \longleftrightarrow \bigsqcup\nolimits_{[L,\omega]_G 
= \fs \in \Omega (G)} \big( \Phi_{e,\zeta_G} (L)^{\fs_L} \q W_\fs \big)_\natural .
\]
Let $\Phi_\cusp (L)$ be the subset of $\Phi_e (L)$ which corresponds to $\Irr_\cusp 
(L)$. Again, its definition depends on Conjecture \ref{conj:1}. The same argument
as above can also be applied to the equivalent formulation \eqref{eq:2.8} of
Conjecture \ref{conj:2}. That leads to a bijection
\begin{equation}\label{eq:3.7}
\Phi_{e,\zeta_G} (G) \longleftrightarrow 
\bigsqcup\nolimits_L \big( \Phi_{\cusp,\zeta_G} (L) \q W(G,L) \big)_\natural , 
\end{equation}
where $L$ runs over the conjugacy classes of Levi subgroups of $G$.

In the upcoming paragraphs we will explain how to reformulate \eqref{eq:3.1} and
\eqref{eq:3.7} entirely in terms of complex reductive groups with Galois actions,
resulting in Conjecture \ref{conj:4}. That and Conjecture \ref{conj:2} should 
form the vertical maps in a commutative, bijective diagram
\begin{equation}\label{eq:3.9}
\xymatrix{
\Irr (G) \ar@{<->}[r]^{LLC} \ar@{<->}[d] & \Phi_{e,\zeta_G} (G) \ar@{<->}[d] \\
\bigsqcup_L \big( \Irr_\cusp (L) \q W(G,L) \big)_\natural \ar@{<->}[r] &
\bigsqcup_L \big( \Phi_{\cusp,\zeta_G} (L) \q W(G,L) \big)_\natural
}
\end{equation}
where both unions run over the same set of represenatives for the conjugacy
classes of Levi subgroups of $G$. The bottom map comes from the LLC for 
supercuspidal $L$-representations, taking desideratum (5) and Proposition 
\ref{prop:3.1} into account. With such a diagram one can try to establish the 
local Langlands correspondence for $G$.
This setup reduces the problem to three more manageable steps:
\begin{itemize}
\item Conjecture \ref{conj:2},
\item Conjecture \ref{conj:4},
\item the LLC for supercuspidal representations.
\end{itemize}
We note that this strategy was already employed to find the LLC for principal 
series representations of split reductive $p$-adic groups \cite[\S 16]{ABPS5}.
In that case the bottom line of the above diagram is a consequence of the
naturality of the LLC for (split) tori.

\subsection{Towards a Galois analogue of the Bernstein theory} \

We would like to rephrase \eqref{eq:3.1} and \eqref{eq:3.7} entirely on
the Galois side. To get started, one has to be able to detect when an enhanced
L-parameter is ``cuspidal'', without knowing the LLC. We note that it is 
impossible to define this properly for L-parameters, since there are L-packets
that contain both supercuspidal and non-supercuspidal representations. 
The enhancement of a L-parameter is essential for its nature.

In view of \cite[D\'efinition 4.11]{Mou}, the correct criterion should be that 
an enhanced L-parameter $(\phi,\rho) \in \Phi_e (G)$ is cuspidal if:
\begin{itemize}
\item $\phi \in \Phi (G)$ is discrete;
\item $\rho \in \Irr (\cS_\phi)$ is cuspidal in the sense of Lusztig's generalized
Springer correspondence \cite{Lus5}.
\end{itemize}
Let $\Phi_{\cusp} (G)$ denote the set of cuspidal (enhanced) L-parameters for $G$.

Furthermore a notion of ``cuspidal support'' of enhanced L-parameters
seems necessary, that is, a well-defined map from $\Phi_e (G)$ to cuspidal
enhanced Langlands parameters of Levi subgroups of $G$.
Such a notion was developed in \cite[\S 4.2.2]{Mou}, and worked
out completely for split classical groups in \cite[\S 4.2.3]{Mou}.

The desiderata of the Langlands correspondence show how ``inertial equivalence''
can be be formulated for L-parameters. Let $\mb I_F$ be the inertia subgroup of
$\mb W_F$ and let $\Fr_F \in \mb W_F$ be a Frobenius element, so that
\[
\mb W_F / \mb I_F \cong \langle \Fr_F \rangle \cong \Z . 
\]
By \cite[(3.3.2)]{Hai} there are natural isomorphisms
\begin{equation}\label{eq:3.8}
X_\nr (G) \cong \big( Z(G^\vee)^{\mb I_F} \big)^\circ_{\langle \Fr_F \rangle} 
\cong H^1_c \Big( \mb W_F / \mb I_F , Z(G^\vee)^{\mb I_F} \big)^\circ \Big) .
\end{equation}
We will denote a typical cuspidal L-parameter by $(\varphi,\varepsilon) \in 
\Phi_\cusp (L) \subset \Phi_e (L)$. In view of Borel's desideratum (2) for 
Conjecture \ref{conj:1}, \cite[5.3.3]{Hai} and \cite[Def.~4.15]{Mou},
we define $(L,\phi,\varepsilon)$, $(L',\varphi',\varepsilon')$ 
to be inertially equivalent (for
$G^\vee$) if there exist $g\in G^\vee$ and $z \in H^1_c \Big( \mb W_F / \mb I_F , 
\big( Z(L^\vee)^{\mb I_F} \big)^\circ \Big)$ such that 
\[
L^{\prime\vee} = {}^g L^\vee ,\;\;\varphi '= z \, {}^g\varphi 
,\;\;\varepsilon'= {}^g \varepsilon .
\]
We denote their inertial equivalence class by $\fs^\vee = 
[L^\vee,\varphi,\varepsilon]_{G^\vee}$, and we let $\Omega^\vee (G)$ 
be the collection of inertial equivalence classes.

The analogue of a Bernstein component in $\Phi_e (G)$ should be
\[
\Phi_e (G)^{\fs^\vee} = \{ (\phi,\rho) \in \Phi_e (G) : \text{ the cuspidal support
of } (\phi,\rho) \text{ lies in } \fs^\vee \} .
\]
Of course this is only meaningful if the cuspidal support of enhanced Langlands
parameters can be defined precisely. We expect that under the LLC 
$\Phi_{e,\zeta_G} (G)^{\fs^\vee}$ will be in bijection with $\Irr (G)^\fs$, where 
$\fs = [L,\omega]_G$ with $\omega \in \Irr_\cusp (L)$ corresponding to some
$(\varphi,\varepsilon) \in \fs^\vee$.

One may wonder how $W(G,L)$ 
acts on $\Phi_{\cusp} (L)$
in \eqref{eq:3.7}. That should come from the action of $N_{G^\vee}(L^\vee \rtimes 
\mb W_F)$ on $\Phi_e (L)$, via the next result.

\begin{prop}\label{prop:3.1}
Let $L$ be any Levi subgroup of $G$. There is a canonical isomorphism
\[
W(G,L) \cong N_{G^\vee}(L^\vee \rtimes \mb W_F) / L^\vee .
\]
\end{prop}
\begin{proof}
First we reformulate $W(G,L)$ in terms of the root datum of $\cG$.
Let $S = \cS (F)$ be a maximal $F$-split torus in $L = \mc L (F)$. The relative
(with respect to $F$) Weyl group of $G = \cG (F)$ is 
\[
W(G,S) = N_G (S) / Z_G (S) .
\]
Both the canonical maps 
\[
(\mathrm{Stab}_{N_G (S)} (L) / Z_G (S) ) \Big/ (N_L (S) / Z_L (S)) \to
\mathrm{Stab}_{N_G (S)}(L) / N_L (S) \to N_G (L) / L
\]
are bijective, the last one because all maximal $F$-split tori in $L$
are $L$-conjugate \cite[Theorem 15.2.6]{Spr}. In other words,
\begin{equation}\label{eq:3.2}
\mathrm{Stab}_{W(G,S)}(L) / W(L,S) \cong W(G,L).
\end{equation}
Let $\cT$ be a maximal $F$-torus of $\mc L$ containing $\cS$. The absolute
Weyl group $W(\cG,\cT) = N_\cG (\cT) / \cT$ is endowed with an action of
$\mb W_F$. The relative Weyl group is the restriction of $W (\cG,\cT)^{\mb W_F}$
to $X^* (\cS)$ \cite[\S 15.3]{Spr}. That is,
\begin{equation}\label{eq:3.3}
W(G,S) \cong W(\cG,\cT)^{\mb W_F} / W (Z_\cG (\cS),\cT)^{\mb W_F} . 
\end{equation}
An element of $N_\cG (\cT)$ normalizes $\mc L$ if and only if it stabilizes the
root subsytem $R(\mc L,\cT) \subset R (\cG,\cT)$. Combining \eqref{eq:3.2} and
\eqref{eq:3.3}, we find
\[
W(G,L) \cong \mathrm{Stab}_{W(\cG,\cT)^{\mb W_F}} (R(\mc L,\cT)) \big/ 
W(\mc L,\cT)^{\mb W_F} .
\]
Now we are in a good position to pass to the complex dual groups. Using
the canonical isomorphism 
\[
W(\cG,\cT) \cong W (\cG^\vee,\cT^\vee) = W (G^\vee,T^\vee),
\]
we obtain 
\begin{equation}\label{eq:3.4}
W(G,L) \cong \mathrm{Stab}_{W(G^\vee,T^\vee)^{\mb W_F}} (R(L^\vee,T^\vee)) \big/ 
W(L^\vee,T^\vee)^{\mb W_F} .
\end{equation}
Because $\cT$ is defined over $F$, $T^\vee$ is $\mb W_F$-stable and we can
form $T^\vee \rtimes \mb W_F$. An element of $N_{G^\vee}(T^\vee)$ is fixed
by $\mb W_F$ if and only if it normalizes $T^\vee \rtimes \mb W_F$.

Inside the Langlands dual group $G^\vee \rtimes \mb W_F$ we can rewrite the
right hand side of \eqref{eq:3.4} as
\begin{multline}\label{eq:3.5}
\big( \mathrm{Stab}_{N_{G^\vee}(T^\vee \rtimes \mb W_F)}(R(L^\vee,T^\vee)) / 
T^\vee \big) \Big/ \big( N_{L^\vee}(T^\vee \rtimes \mb W_F) / T^\vee \big) \\
\cong \mathrm{Stab}_{N_{G^\vee}(T^\vee \rtimes \mb W_F)}(R(L^\vee,T^\vee))  
\Big/ \big( N_{L^\vee}(T^\vee \rtimes \mb W_F) .
\end{multline}
A standard argument shows that the canonical injection
\begin{equation}\label{eq:3.6}
\mathrm{Stab}_{N_{G^\vee}(T^\vee \rtimes \mb W_F)}(R(L^\vee,T^\vee))  
\Big/ \big( N_{L^\vee}(T^\vee \rtimes \mb W_F) \big) \to 
N_{G^\vee}(L^\vee \rtimes \mb W_F) / L^\vee
\end{equation}
is surjective. Namely, for $n \in N_{G^\vee}(L^\vee \rtimes \mb W_F)$, 
$n T^\vee n^{-1}$ is a maximal torus of the complex group $L^\vee$, 
so it is conjugate to $T^\vee$ by some $l \in L^\vee$. Then 
\[
l n \in N_{G^\vee}(L^\vee \rtimes \mb W_F) \cap N_{G^\vee}(T^\vee \rtimes \mb W_F) 
= \mathrm{Stab}_{N_{G^\vee}(T^\vee \rtimes \mb W_F)}(R(L^\vee,T^\vee)) .
\]
Hence $W(G,L)$ is canonically isomorphic to the right hand sides of \eqref{eq:3.5}
and \eqref{eq:3.6}.
\end{proof}

Now we have a well-defined action of $W(G,L) \cong N_{G^\vee}(L^\vee \rtimes 
\mb W_F) / L^\vee$ on
\[
\Phi_{\cusp} (L) = \bigsqcup\nolimits_{\fs^\vee_L = [L^\vee,\varphi,\varepsilon]_{L^\vee}}
\Phi_e (L)^{\fs_L^\vee} .
\]
The action preserves this decomposition because it stabilizes the group of
unramified characters $X_\nr (L) \cong H^1_c \Big( \mb W_F / \mb I_F , 
(Z(L^\vee)^{\mb I_F} \big)^\circ \Big)$. Hence we can transfer the definition of
Bernstein's finite group $W_\fs$ to the Galois side. For $\fs^\vee = [L^\vee,
\varphi,\varepsilon]_{G^\vee} \in \Omega^\vee (G)$ we define:
\[
W_{\fs^\vee} = \text{stabilizer of } \Phi_e (L)^{\fs_L^\vee} \text{ in }
N_{G^\vee}(L^\vee \rtimes \mb W_F) / L^\vee.
\]
It is expected (and proved in \cite[Th\'eor\`eme 5.6]{Mou} in the case of split groups of 
classical type) that if $\sigma\in\Irr(L)$ corresponds to $(\varphi,\varepsilon)\in
\Phi_{\cusp}(L)$ via LLC then the groups $W_\fs$ and $W_{\fs^\vee}$ are isomorphic.

\subsection{Langlands parameters and extended quotients} \ 

It is reasonable to expect that the conjectural bijection
\[
\Phi_e (G)^\fs \longleftrightarrow \big( \Phi_e (L)^{\fs_L} \q W_\fs \big)_\natural
\]
from \eqref{eq:3.1} can be constructed purely in terms of Langlands parameters, 
without using $p$-adic groups. Indeed, this was already done for $\GL_n (F)$ in
\cite[\S 1]{BrPl}. Let us give two more examples.\\

\textbf{Example.}
Let $G = \SL_2 (F), G^\vee = \PGL_2 (\C)$ and $L = T \cong F^\times$. 
We record that $W(G^\vee,T^\vee) \cong W_\fs \cong \Z / 2 \Z$, where $\Irr^\fs (G)$ 
is Iwahori--spherical Bernstein component. For $\phi$ we simply take the unit map 
$\mb W_F \times \SL_2 (\C) \to T^\vee \subset G^\vee$. Then $\phi \in \Phi_\cusp (T)$
and $(\phi,\mathrm{sign}_{W_\fs}) \in \big( \Phi_e (T)^{\fs_T} \q W_\fs \big)_2$.

From this we want to construct  $(\tilde \phi,\rho) \in \Phi_e (G)^\fs$. The Springer
correspondence for $W_\fs$ associates to the sign representation the conjugacy class
of the unipotent element $u = \matje{1}{1}{0}{1} \in \PGL_2 (\C)$. We define $\tilde 
\phi$ by 
\[
\tilde \phi \big|_{\mb W_F} = \phi \big|_{\mb W_F} = 1 \quad \text{and} \quad  
\tilde \phi \Big(1,\matje{1}{1}{0}{1}\Big) = u .
\]
For a lack of choice we have to take $\rho = 1$.
Notice that this agrees with the example on page \pageref{casesABPS} and with the 
LLC for $\SL_2 (F)$: both $(\phi,\mathrm{sign}_{W_\fs})$ and $\tilde \phi$ correspond 
to the Steinberg representation. \\

\textbf{Example.}
Let $G = \GL_m (D)$ and let $\chi \in \Irr (\SL_{md}(\C))$ be the character that 
defines $G$ as an inner twist of  $\GL_{md}(F)$ (see page \pageref{prop:Kot}).
Assume that $\phi$ is a Langlands parameter for a supercuspidal representation of a 
standard Levi subgroup $L = \prod_i \GL_{m_i}(F)^{e_i}$ of $G$, of the form
$\prod_i \phi^{e_i}$ with $\phi_i : \mb W_F \times \SL_2 (\C) \to \GL_{m_i d}(\C)$
discrete. 

Since $\cR_\phi = 1$ for all $\phi \in \Phi (\GL_{md}(F))$, and by Lemma
\ref{lem:1.1}, we have $\cZ_\phi = \cS_\phi$ and $(\phi,\chi) \in \Phi_e (L)^{\fs_L^\vee}$.
The stabilizer of $\phi$ in $W_{\fs^\vee}$ is $W_{\fs^\vee,\phi} \cong \prod_i S_{e_i}$. 
Let $\rho \in \Irr (W_{\fs^\vee,\phi})$, so that 
$(\phi,\chi,\rho) \in \big( \Phi_e (L)^{\fs_L^\vee} \q W_\fs \big)_2$. 

To construct an element of $\Phi_e (G)^{\fs^\vee}$ from this we proceed as 
above, only with more data. Via the Springer correspondence for $W_{\fs,\phi}$, $\rho$
determines a unipotent class $[u]$ in $Z_{\GL_{md}(\C)}(\phi) \cong 
\prod_i \GL_{e_i}(\C)$. We put 
\[
\tilde \phi \big|_{\mb W_F} = \phi \big|_{\mb W_F} \quad \text{and} \quad 
\tilde \phi \Big(1,\matje{1}{1}{0}{1}\Big) = u \, \phi \Big(1,\matje{1}{1}{0}{1}\Big) .
\]
Then $(\tilde \phi,\chi) \in \Phi_e (G)^{\fs^\vee}$ and with \cite[Theorem 5.3]{ABPS6} one 
can check that it corresponds to the same $G$-representation as $(\phi,\chi,\rho)$.\\

With all the notions from the previous paragraph we can formulate a Galois version of 
Conjecture \ref{conj:2}, see \cite[\S 5.3]{Mou}.

\begin{conj}\label{conj:4}
Let $L$ be any Levi subgroup of $G$ and let $\fs^\vee = [L^\vee,\varphi,\varepsilon]_{G^\vee} \in
\Omega^\vee (G)$. There exists a family of 2-cocycles $\natural$ and bijections
\[
\begin{array}{ccc}
\Phi_{e,\zeta_G} (G)^{\fs^\vee} & \longleftrightarrow & 
\big( \Phi_{e,\zeta_G} (L)^{\fs^\vee_L} \q W_{\fs^\vee} \big)_\natural , \\
\bigsqcup\nolimits_{\fs^\vee = [L^\vee,\varphi,\varepsilon]_{G^\vee}} 
\Phi_{e,\zeta_G} (G)^{\fs^\vee} 
& \longleftrightarrow & \big( \Phi_{\cusp,\zeta_G} (L) \q W(G,L) \big)_\natural .
\end{array}
\]
Moreover these maps preserve boundedness, and they can be constructed entirely in 
terms of complex reductive groups with $\mb W_F$-actions.
\end{conj}

This conjecture was proven for split classical groups in \cite[Th\'eor\`eme 5.5]{Mou} 
and for principal series Bernstein components of split reductive groups in 
\cite[\S 5 and Theorem 8.2]{ABPS5}.

We note that the two bijections in Conjecture \ref{conj:4} are the same, since by
the definition of $W_{\fs^\vee}$ the canonical map
\[
\bigsqcup\nolimits_{\fs^\vee = [L^\vee,\varphi,\varepsilon]_{G^\vee}} 
\big( \Phi_{e,\zeta_G} (L)^{\fs^\vee_L} \q W_{\fs^\vee} \big)_\natural 
\longrightarrow \big( \Phi_{\cusp,\zeta_G} (L) \q W(G,L) \big)_\natural 
\]
is a bijection. It seems that Conjecture \ref{conj:4} uses the $p$-adic groups
$G$ and $L$, but this is only notational. All the relevant objects are defined
in terms of ${}^L G$, the character $\zeta_G$ of $Z(G^\vee_\sc)$, and the Levi 
subgroup $L^\vee \rtimes \mb W_F \subset {}^L G$.

\section{Topological K-theory}

We discuss the K-theory of the reduced $C^*$-algebra of $G$. Different pictures
of these groups are provided by several conjectures: the Baum--Connes conjecture,
Conjecture \ref{conj:2} and the local Langlands correspondence (although only
in an heuristic way).

\subsection{Equivariant K-theory} \

This paragraph is a counterpart to paragraph \ref{par:extquot}. We work in the
same generality, just with groups acting on nice spaces, and we end up with the
topological K-theory of extended quotients.

Let $X$ be a locally compact Hausdorff space and let $\Gamma$ be a group acting
on $X$. For simplicity we assume that $\Gamma$ is finite. The $\Gamma$-equivariant
K-theory of $X$ was defined in \cite[\S 2.4]{Ati}. When $X$ is compact, 
$K_\Gamma^0 (X)$ is the Grothendieck group of the semigroup of complex 
$\Gamma$-vector bundles on $X$. When $X$ is only locally compact, we let 
$X \cup \{ \infty \}$ be its one-point compactification, and we put
\begin{equation}\label{eq:4.1}
K_\Gamma^0 (X) = \ker \big( K_\Gamma^0 (X \cup \{ \infty \}) \to 
K_\Gamma^0 (\{ \infty \}) \big) .
\end{equation}
The equivariant $K^1$-group is defined via the suspension functor. It can be
expressed as 
\[
K_\Gamma^1 (X) = K_\Gamma^0 (X \times \R) ,
\]
where $\Gamma$ acts trivially on $\R$. Typically one writes
\[
K_\Gamma^* (X) = K_\Gamma^0 (X) \oplus K_\Gamma^1 (X) ,
\]
a $\Z / 2 \Z$-graded abelian group. Let 
\[
C_0 (X) = \{ f \in C (X \cup \{\infty\},\C) : f(\infty) = 0 \} 
\]
be the commutative $C^*$-algebra of functions on $X$ which vanish at infinity.
By the Serre--Swan Theorem its K-theory is
\[
K_* (C_0 (X)) \cong K^* (X) . 
\]
The group $\Gamma$ acts on $C_0 (X)$ by automorphisms, and we form the crossed product
$C_0 (X) \rtimes \Gamma$. Recall from Lemma \ref{lem:B.6} that $\Irr (C_0 (X) \rtimes
\Gamma) \cong (X \q \Gamma )_2$.
By the Green--Julg Theorem \cite{Jul} and the equivariant
Serre--Swan Theorem \cite[2.3.1]{Phi} there is a natural isomorphism
\begin{equation}\label{eq:4.2} 
K_* (C_0 (X) \rtimes \Gamma) \cong K_\Gamma^* (X) . 
\end{equation}
Thus we can interpret $K_\Gamma^* (X)$ as the K-theory of the topological space
$(X \q \Gamma )_2$. Of course that space is usually not Hausdorff, so the statement
is not precise, it is rather a manifestation of the philosophy of noncommutative
geometry.

Now we consider twisted extended quotients. Let $\natural : \Gamma \times \Gamma
\to \C^\times$ be a 2-cocycle. As in \eqref{eq:2.9}, we can find a central
extension
\begin{equation}\label{eq:4.3}
1 \to Z \to \tilde \Gamma \to \Gamma \to 1 , 
\end{equation}
a character $\chi_\natural$ of $Z$ and a minimal idempotent $p_\natural \in \C [Z]$
such that $p_\natural \C [\tilde \Gamma] \cong \C [\Gamma,\natural]$.

The group $\tilde \Gamma$ also acts on $X$, via its projection to $\Gamma$.
Then $C_0 (X) \rtimes p_\natural \C [\tilde \Gamma]$ is a direct summand of
$C_0 (X) \rtimes \tilde \Gamma = C_0 (X) \rtimes \C [\tilde \Gamma]$.
It follows from \eqref{eq:4.2} that
\begin{equation}\label{eq:4.4}
K_* (C_0 (X) \rtimes p_\natural \C [\tilde \Gamma] ) \cong
p_\natural K^*_{\tilde \Gamma}(X) .
\end{equation}
In view of Lemma \ref{lem:B.1}, the left hand side can be regarded as the
K-theory of the topological space $(X \q \Gamma )_\natural$. The right hand side
of \eqref{eq:4.4} also admits a geometric interpretation. We saw in \eqref{eq:4.1}
that $K^0_{\tilde \Gamma}(X)$ is built from $\tilde \Gamma$-vector bundles on $X$.
The central idempotent $p_\natural$ selects the direct summands corresponding to
the $\tilde \Gamma$-vector bundles on which $Z$ acts as $\chi_\natural$. Similarly,
$K^1_{\Gamma}(X)$ can be constructed from the semigroup of $\tilde \Gamma$-vector
bundles on $X \rtimes \R$ on which $Z$ acts as $\chi_\natural$. These semigroups
of vector bundles depend on $X, \Gamma$ and $\natural$, but not on the central
extension $\tilde \Gamma$ chosen to analyse $\natural$. Thus we can define the
$\natural$-twisted $\Gamma$-equivariant $K$-theory of $X$ as 
\[
K^*_{\Gamma,\natural}(X) := p_\natural K^*_{\tilde \Gamma}(X) . 
\]
Then, loosely speaking,
\begin{equation}\label{eq:4.5}
K^* ( (X \q \Gamma)_\natural ) \cong K^*_{\Gamma,\natural}(X) . 
\end{equation}

\subsection{The Baum--Connes conjecture} \

As before, let $G = \cG (F)$ be a reductive $p$-adic group. The reduced $C^*$-algebra
$C_r^* (G)$ is the completion of $\cH (G)$ in the algebra of bounded linear operators
on the Hilbert space $L^2 (G)$. It follows from the work of Harish--Chandra 
(see \cite[\S 10]{Vig}) that the irreducible representations of $C_r^* (G)$ can be
identified with those of the Schwartz algebra of $G$. By \cite[\S III.7]{Wal} the
latter are the same as irreducible tempered $G$-representations. Thus we get
\begin{equation}\label{eq:4.7}
\Irr (C_r^* (G)) = \Irr_\temp (G) ,
\end{equation}
which means that $C_r^* (G)$ is the correct $C^*$-algebra to study the noncommutative
geometry of the tempered dual of $G$.
The structure of $C_r^* (G)$ was described by means of the Fourier transform in
\cite{Ply}.

The Baum--Connes conjecture provides a picture of the $K$-theory of this $C^*$-algebra
in geometric terms. Let $\cB (G)$ be the (nonreduced) affine building of $G$, as
developed by Bruhat and Tits \cite{BrTi1,BrTi2}.
This is a proper $G$-space with many remarkable properties, for example:
\begin{itemize}
\item $\cB (G)$ satisfies the negative curvature inequality \cite[2.3]{Tit} and
hence is contractible and has unique geodesics \cite[\S VI.3]{Bro};
\item every compact subgroup of $G$ fixes a point of $\cB (G)$, see 
\cite[\S 2.3.1]{Tit} or \cite[\S VI.4]{Bro}.
\end{itemize}
In view of \cite[Proposition 1.8]{BCH}, these properties make $\cB (G)$ into a
universal space for proper $G$-actions \cite[Definition 1.6]{BCH}. 

The $G$-equivariant K-homology $K_*^G (\cB (G))$ of the building was defined in
\cite[\S 3]{BCH}. The Baum--Connes conjecture asserts that the canonical assembly map
\begin{equation}\label{eq:4.6}
K_*^G (\cB (G)) \to K_* (C_r^* (G)) 
\end{equation}
is an isomorphism. This was proven (for a large class of groups containing
$G$) in \cite{Laf1}. For the groups under consideration the Baum--Connes conjecture
can also be formulated and proven more algebraically \cite{HiNi,Schn}, with 
equivariant cosheaf homology (also known as chamber homology) \cite[\S 2]{ABP}.
By \cite{Sol1} these two versions of the conjecture are compatible.

The left-hand-side of \eqref{eq:4.6}, defined in terms of $K$-cycles, has never been 
directly computed for a noncommutative reductive $p$-adic group. Results of Voigt 
\cite{Voi} allow us to replace the left-hand-side with the 
chamber homology groups. Chamber homology has been directly computed for only two 
noncommutative $p$-adic groups:  $\SL_2 (F)$ \cite{BHP1} and $\GL_3 (F)$ \cite{AHP}.   
In the case of $\GL_3 (F)$, one can be sure that representative cycles in all the 
homology groups have been constructed only by checking with the right-hand-side 
of the Baum--Connes conjecture.
In other words, one always has to have an independent computation of the right-hand-side.  

On the $C^*$-algebra of \eqref{eq:4.6} side our earlier conjectures have something to say. 
The Bernstein decomposition of $\cH (G)$ \eqref{eq:2.10} gives rise to a factorization
\[
C_r^* (G) = \prod\nolimits_{\fs \in \Omega (G)} C_r^* (G)^\fs  \quad
\text{with} \quad \Irr (C_r^* (G)^\fs) = \Irr_\temp (G)^\fs .
\]
Morally speaking, $K_* (C_r^* (G)^\fs)$ is the K-theory of the topological
space $\Irr_\temp (G)^\fs$. Combining this with Conjecture \ref{conj:2} and 
\eqref{eq:4.5} leads to:

\begin{conj}\label{conj:5}
Let $\fs \in \Omega (G)$. There exists a canonical isomorphism
\[
K^*_{W_\fs,\natural}(T_{\fs,un}) \to K_* (C_r^* (G)^\fs) . 
\]
\end{conj}

This is the topological K-theory version of Conjecture \ref{conj:2}. Of course
it is much weaker, since it only says something about the cohomology of
$(T_\fs \q W_\fs )_\natural$, and not so much about the space itself. Yet in
practice, with some additional knowledge of the underlying algebras, this
already provides a lot of information. Conjecture \ref{conj:5} provides a much 
finer and more precise formula for $K_*(C^*_r(G))$ than Baum--Connes alone. 

Let us consider the reduced Iwahori-spherical $C^*$-algebra  $C_r^*(G)^{\mathfrak{i}} 
\subset C^*_r(G)$ in more detail. The primitive ideal spectrum of $C^*_r(G)^{\mathfrak{i}}$ 
can be identified with the irreducible tempered representations of $G$ which admit nonzero 
Iwahori-fixed vectors. We assume that $G$ is split, so $\mf i = [T,1]_G$ and 
$T_{\mf i} = T^\vee$ is a maximal torus in the complex dual group $G^\vee$.
In this special case,  Conjecture \ref{conj:5} asserts that 
\begin{align}\label{KKK}
K_j(C_r^*(G)^{\mathfrak{i}}) \cong K^j_{W_{\mf i}}(T^\vee_{un}) 
\end{align}
with $j = 0,1$. Here $K^j_{W_{\mf i}}(T^\vee_{un})$ is the classical topological equivariant 
K-theory for the Weyl group $W_{\mf i} \cong W (G^\vee ,T^\vee)$ acting on the compact torus 
$T^\vee_{un} = T_{\mathfrak{i}, un}$.

Let $X^*(T^\vee_{un})$ denote the group of Lie group morphisms from $X^*(T^\vee_{un})$ to 
U$(1)$, that is, $X^*(T^\vee_{un})$ denotes the Pontryagin dual of $T^\vee_{un}$. 
It is naturally isomorphic with the lattice of algebraic characters of $T^\vee$. We have 
\begin{align*}
C^*_r (X^*(T^\vee_{un}) \rtimes W_{\mf i}) \cong C(T^\vee_{un}) \rtimes W_{\mf i}
\end{align*}
by a standard Fourier transform.
By \eqref{eq:2.11} and \cite[Theorem 5.1.4]{Sol2}
\begin{align*}
K_j (C_r^*(G)^{\mathfrak{i}}) \otimes_{\Z}\Q \cong
K_j ( C^*_r (X^* (T^\vee_{un}) \rtimes W_{\mf i}))\otimes_{\Z}\Q ,
\end{align*}
where $j = 0,1$. With \eqref{eq:4.2} we get
\begin{align}\label{Sol}
K_j (C_r^*(G)^{\mathfrak{i}})\otimes_{\Z}\Q \cong 
K^j_{W_{\mf i}} (T^\vee_{un})\otimes_{\Z}\Q ,
\end{align}
which establishes \eqref{KKK} modulo torsion.\\

In general, if Conjecture \ref{conj:3} would hold for $\fs$, then $C_r^* (G)^\fs$
would be Morita equivalent with $e_\fs C_r^* (G) e_\fs$, and that algebra could
be described in terms of $C^*$-completions of affine Hecke algebras \cite{Opd}.
With the techniques developed in \cite[\S 5.1]{Sol1} and \cite[\S 6]{ABPS6} that
would go a long way towards Conjecture \ref{conj:5}.

Now two pictures of $K_* (C_r^* (G))$ are available, namely $K_*^G (\cB (G))$ and\\ 
$\bigoplus_{\fs \in \Omega (G)} K^*_{W_\fs,\natural}(T_{\fs,un})$. 
Unfortunately they are not compatible in any obvious way. It is even unclear
how a Bernstein decomposition of $K_*^G (\cB (G))$ would look like, see 
\cite[\S 5]{BHP} for a discussion of the analogous problem in chamber homology.

We sketch how some comparisons can be made. Let $S$ be a maximal $F$-split torus
of $G$ and let $A_S = X_* (S) \otimes_\Z \R$ be the corresponding apartment of
$\cB (G)$. It is endowed with an action of 
\[
W^e (G,S) := N_G (S) / Z_G (S)_\cpt \cong Z_G (S) / Z_G (S)_\cpt \rtimes W(G,S) ,
\]
a group which contains $X_* (S) \rtimes W(G,S)$ as a subgroup of finite index.
More generally, for any Levi subgroup $L \subset G$ the group $W^e (L,S)$ acts
on $A_S$, and $A_S$ is a universal example for proper $W^e (L,S)$-actions.

Let $\fs = [L,\omega]_G$ and regard $\omega$ as the basepoint of $T_{\fs,un}$. 
Via \eqref{eq:2.1} this turns $T_{\fs,un}$ into a Lie group, so we can speak
of its characters. Then
\begin{equation}\label{eq:4.8}
K^*_{W_\fs,\natural}(T_{\fs,un}) \cong 
K_* \big( C (T_{\fs,un}) \rtimes \C [W_\fs,\natural] \big) \cong
K_* \Big( C_r^* (X^* (T_{\fs,un})) \rtimes \C [W_\fs,\natural] \Big) .
\end{equation}
Choosing a central extension as in \eqref{eq:4.3}, we can rewrite the right
hand side of \eqref{eq:4.8} as 
\[
K_* \Big( p_\natural C_r^* \big( X^* (T_{\fs,un}) \rtimes 
\tilde{W_\fs}\big) \Big) = p_\natural K_* \Big( C_r^* \big( X^* (T_{\fs,un}) 
\rtimes \tilde{W_\fs}\big) \Big) . 
\]
It follows from \cite[\S V.2.6]{Ren} that the character lattice $X^* (T_{\fs,un}) 
= X^* (T_\fs)$ is naturally isomorphic to a cocompact subgroup of 
\[
A_{Z(L)} = \big( Z(L) / Z(L)_\cpt \big) \otimes_\Z \R ,
\]
So the Baum--Connes conjecture for $X^* (T_{\fs,un}) \rtimes \tilde{W_\fs}$ 
gives isomorphisms
\begin{equation}\label{eq:4.9}
K^*_{W_\fs,\natural}(T_{\fs,un}) \cong p_\natural K_* \Big( C_r^* 
\big( X^* (T_{\fs,un}) \rtimes \tilde{W_\fs}\big) \Big) \cong
p_\natural K_*^{X^* (T_{\fs,un}) \rtimes \tilde{W_\fs}} (A_{Z(L)}) .
\end{equation}
The canonical embedding $A_{Z(L)} \to A_S \subset \cB (G)$ should identify
the isotropy groups for the action of $X^* (T_{\fs,un}) \rtimes W_\fs$ acting
on $A_{Z(L)}$ with subquotients of the isotropy groups of $G$ acting on 
$\cB (G)$. Via \eqref{eq:4.9} that should give a map
\[
K^*_{W_\fs,\natural}(T_{\fs,un}) \to K_*^G (\cB (G)). 
\]
Of course this construction is too simple, because it does not take the 
$L$-representation $\omega$ into account. Yet at least it gives us a geometric
idea of how the two pictures of $K_* (C_r^* (G))$ can be related. Probably a
good map from $K_* (C_r^* (G)^\fs)$ to $K_*^G (\cB (G))$ will have to
involve a nice idempotent for $\fs$. But, even with a $\fs$-type available
the issue is currently unclear. 

\subsection{Relations with the LLC} \

Recall that Conjecture \ref{conj:1} predicts a bijection
\[
\Irr_\temp (G) \longleftrightarrow \Phi_{e,\zeta_G,\bdd}(G) , 
\]
where the subscript ``bdd'' indicates bounded L-parameters. By \eqref{eq:4.7}
$K_* (C_r^* (G))$ can be regarded as the K-theory of $\Irr_\temp (G)$. It is not 
so clear what the (topological) K-theory of $\Phi_{e,\zeta_G,\bdd}(G)$ should be, 
because there is no convenient algebra in sight. Here Conjecture \ref{conj:4} is 
useful. As it respects boundedness of L-parameters, it predicts a bijection
\[
\Phi_{e,\zeta_G,\bdd}(G) \longleftrightarrow 
\bigsqcup_{\fs^\vee = [L^\vee,\varphi,\varepsilon]_{G^\vee} \in \Omega^\vee (G)} 
\big( \Phi_{e,\zeta_G,\bdd}(L)^{\fs^\vee_L} \q W_{\fs^\vee} \big)_\natural .
\]
The K-theory of the right hand side can be interpreted with \eqref{eq:4.5}.
The LLC for $\Irr_\cusp (L)$ should provide a $W(G,L)$-equivariant bijection
\[
\Irr_{\cusp,\temp}(L) \longleftrightarrow \Phi_{\cusp,\zeta_G,\bdd}(L) , 
\]
which induces a group isomorphism $W_\fs \cong W_{\fs^\vee}$ if $T_\fs$ 
corresponds to $\Phi_{e,\zeta_G} (L)^{\fs^\vee_L}$. 

Let us combine all these descriptions of $K_* (C_r^* (G))$ in one diagram:
\[
\xymatrix{ 
K_* (C_r^* (G)) \ar@{=}[d] \ar@{<->}[rr]^{\text{Baum--Connes}} & & 
K_*^G (\cB (G)) \ar@{<.>}[d] \\ 
K_* (C_r^* (G)) \ar@{<->}[rr]^{\text{Conjecture \ref{conj:5}}} \ar@{<->}[dd] & & 
\bigoplus_{\fs \in \Omega (G)} K^*_{W_\fs,\natural}(T_{\fs,un}) 
\ar@{<->}[d]^{\text{cuspidal LLC}} \\
& & \bigoplus_{\fs^\vee = [L^\vee,\varphi,\varepsilon]_{G^\vee} \in 
\Omega^\vee (G)} K^*_{W_{\fs^\vee},\natural} 
\Big( \Phi_{e,\zeta_G,\bdd}(L)^{\fs^\vee_L} \Big) 
\ar@{<->}[d]^{\text{Conjecture \ref{conj:4}}} \\
``K^* (\Irr_\temp (G))" \ar@{<->}[rr]^{``K^* (\text{LLC})"} & &
``K^* (\Phi_{e,\zeta_G,\bdd}(G))"
}\]
On the top of the right hand side we have the ``$p$-adic'' geometry of the
Bruhat--Tits building of $G$, combined with the noncommutative geometry from
equivariant K-homology. At the bottom we find, in some sense, the cohomology of
the space of enhanced bounded L-parameters for $G$. The extended quotients 
obtained from the Bernstein decomposition for $G$ interpolate between these very
different settings. In this way our Conjectures \ref{conj:4} and \ref{conj:5}
connect the Baum--Connes conjecture and the local Langlands
correspondence.

\end{document}